\documentclass{article}[12pt]
\usepackage{amssymb}
\usepackage{amsmath,amsthm}
\usepackage[latin1]{inputenc}
\usepackage{graphicx}
\usepackage{hyperref}
\usepackage{enumerate}
\usepackage{tikz}
\usepackage{mathrsfs}
\usepackage{verbatim}
\usepackage{color,url}
\usepackage{subfigure}

\hypersetup{colorlinks=true, linkcolor=blue, citecolor=blue, urlcolor=blue}

\newtheorem{remark}{Remark}[section]

\newtheorem{theorem}[remark]{Theorem}
\newtheorem{proposition}[remark]{Proposition}
\newtheorem{corollary}[remark]{Corollary}

\newtheorem{conjecture}[remark]{Conjecture}

\newcommand{\wdim}{\operatorname{\mathrm{wdim}}}
\newcommand{\cp}{\,\square\,}

\begin{document}

\title{On the weak $k$-metric dimension of Hamming graphs}

\author{Elena Fern\'andez$^{a,}$\thanks{\texttt{elena.fernandez@uca.es}}
\and Sandi Klav\v zar$^{b,c,d,}$\thanks{\texttt{sandi.klavzar@fmf.uni-lj.si}}
\and Dorota Kuziak$^{a,}$\thanks{\texttt{dorota.kuziak@uca.es}}
\and Manuel Mu\~noz-M\'arquez$^{a,}$\thanks{\texttt{manuel.munoz@uca.es}}
\and Ismael G. Yero$^{e,}$\thanks{\texttt{ismael.gonzalez@uca.es}}
}
\maketitle

\begin{center}
$^a$ Departamento de Estad\'istica e Investigaci\'on Operativa, Universidad de C\'adiz, Spain \\

\medskip
$^b$ Faculty of Mathematics and Physics, University of Ljubljana, Slovenia\\
\medskip
	
$^c$ Institute of Mathematics, Physics and Mechanics, Ljubljana, Slovenia\\
\medskip
	
$^d$ Faculty of Natural Sciences and Mathematics, University of Maribor, Slovenia\\
\medskip

$^e$ Departamento de Matem\'{a}ticas, Universidad de C\'adiz, Algeciras Campus, Spain \\
\medskip
\end{center}

\begin{abstract}
Given a connected graph $G$, a set of vertices $X\subset V(G)$ is a weak $k$-resolving set of $G$ if for each two vertices $y,z\in V(G)$, the sum of the values $|d_G(y,x)-d_G(z,x)|$ over all $x\in X$ is at least $k$, where $d_G(u,v)$ stands for the length of a shortest path between $u$ and $v$. The cardinality of a smallest weak $k$-resolving set of $G$ is the weak $k$-metric dimension of $G$, and is denoted by $\mathrm{wdim}_k(G)$. In this paper, $\mathrm{wdim}_k(K_n\,\square\,K_n)$ is determined for every $n\ge 3$ and every $2\le k\le 2n$. An improvement of a known integer linear programming formulation for this problem is developed and implemented for the graphs $K_n\,\square\,K_m$. Conjectures regarding these general situations are posed.
\end{abstract}


\noindent
{\bf AMS Subj.\ Class.\ (2020)}:  05C12, 05C76, 90C05

%%%%%%%%%%%%%%%%%%%%%%%%%%%%%%%
\section{Introduction}
%%%%%%%%%%%%%%%%%%%%%%%%%%%%%%%

The area of metric dimension related parameters in graphs has been a very active one in the last two decades, although its notion dates back to about 70 years ago when the related concept was introduced for general metric spaces in \cite{Blumenthal}. For the specific case of graphs, the first information on this topic are coming from the 1970's due to Slater \cite{Slater1975}, and independently, also by Harary and Melter \cite{Harary1976}. This topic attracted several investigation in in various directions including combinatorial, computational, and applied. For instance, an interesting application appeared in \cite{Tillquist-2019}, where the authors designed some sort of methodology for embedding biological sequence data into Hamming graphs. To do so, they applied some metric dimension notions. The obtained embedding was further used in machine learning algorithms that learn classifiers from such datasets. Some other recent works on the classical metric dimension of graphs are for instance \cite{Bailey-2023,Dankelmann-2023,foster-2024,Tapendra}. In addition, for more information on this concept and related ones, we suggest the two surveys \cite{Kuziak,Tillquist}.

One of the most common developments concerning the metric dimension of graphs relates to describing different variations of the concept in order to give more insight into the classical concept, or to better understand some practical situations in which extra properties are needed. The compendium \cite{Kuziak} surveys a large number of these variations, and the main contributions about each of them. Very recently, a variation called weak $k$-metric dimension was presented in \cite{Peterin}, which is an attempt to soften the more restrictive notion of $k$-metric dimension of graphs, already known from \cite{Estrada-Moreno2013}.

Throughout our whole exposition $G=(V(G),E(G))$ represents a connected undirected graph without loops and multiple edges.
Given three vertices $x,y,z\in V(G)$, it is said that
$$\Delta_z(x,y)=|d_G(x,z)-d_G(y,z)|\,,$$
where the notation $d_G(a,b)$ stands for the number of edges on a shortest $a,b$-path in $G$, i.e., the distance between $a$ and $b$.
Consider a set $S\subseteq V(G)$ and an integer $k\ge 1$.
The set $S$ is known as a \emph{weak $k$-resolving set} for $G$ if it is satisfied that $$\sum_{w\in S}\Delta_w(x,y)\ge k$$ for each two vertices $x,y\in V(G)$.
The \emph{weak} $k$-\emph{metric dimension} of $G$, written $\wdim_{k}(G)$, is the cardinality of a smallest weak $k$-resolving set of $G$.
Any weak $k$-resolving set having cardinality equal to $\wdim_{k}(G)$ is called a \emph{weak $k$-metric basis} for $G$.
The concepts above were recently defined in \cite{Peterin}.
It is clear that a graph $G$ does not have weak $k$-resolving sets for every integer $k$.
In this sense, by $\kappa(G)$ we represent the largest integer $k$ such that $G$ contains a weak $k$-resolving set.
In addition, it is also said that a graph $G$ is {\em weak $\kappa(G)$-metric dimensional}.

The metric dimension of $K_n\cp K_n$ was studied in \cite{Caceres-2007}, where the formula $\dim(K_n\cp K_n)=\left\lfloor\frac{4n-2}{3}\right\rfloor$ was proved, as a part of a more general result. On the other hand, it is known from \cite[Corollary 2]{Peterin} that $\dim(G)=\wdim_1(G)$ for any graph $G$. Thus, in view of these comments,
$$\wdim_1(K_n\cp K_n)=\left\lfloor\frac{4n-2}{3}\right\rfloor\,.$$
In this paper we complement this result, by determining $\wdim_k(K_n\cp K_n)$ for any integer $n\ge 2$ and any feasible $k\ge 2$. In fact, we prove the following formula.

\begin{theorem}
\label{thm:wdim-Kn-Kn-all}
If $n\ge 3$ and $2\le k\le 2n$, then
\begin{equation*}
\wdim_{k}(K_n\cp K_n) =
\begin{cases}
    \left\lceil \frac{4n}{3}\right\rceil; & \mbox{if $k=2$}\,,\\[0.2cm]
    n\left\lceil\frac{k}{2}\right\rceil; & \mbox{if $k=3$ or $k$ is even}\,, \\[0.2cm]
    n\left\lceil\frac{k}{2}\right\rceil-1; & \mbox{otherwise}\,.
\end{cases}
\end{equation*}
\end{theorem}

In addition, an integer linear programming formulation for this problem, known from \cite{Peterin}, is improved and implemented for the graphs $K_n\,\square\,K_m$. Conjectures regarding these general situations are posed.

%%%%%%%%%%%%%%%%%%%%%%%%%%%%%%%
\section{Preliminaries}
\label{sec:preliminaries}
%%%%%%%%%%%%%%%%%%%%%%%%%%%%%%%

Unless stated otherwise, all graphs considered are connected.  If $G$ is a graph, $S\subseteq V(G)$, and $x,y\in V(G)$, then let
$$\Delta_S(x, y) = \sum_{s\in S} \Delta_s(x, y)\,.$$
If $S = V(G)$, we simplify the notation $\Delta_{V(G)}(x, y)$ to $\Delta(x, y)$. Having this notation, we can recall the following fundamental fact.

\begin{proposition} {\rm \cite[Observation 5]{Peterin}}
\label{prop:kappa=min-Delta}
If $G$ is a graph, then $$\kappa(G) = \min \{ \Delta(x, y):\ x, y \in V(G), x\ne y\}\,.$$
\end{proposition}

Let $G$ and $H$ be any (connected) graphs, and $G\cp H$ be their Cartesian product, which is a graph defined on the vertex set $V(G)\times V(H)$, and having edges $(g,h)(g',h')$ if either $g=g'$ and $hh'\in E(H)$; or $gg'\in E(G)$ and $h=h'$. Throughout the paper, for the complete graph $K_n$, we will adopt the convention $V(K_n) = \mathbb{Z}_n$ and hence, $V(K_n\cp K_m)=\mathbb{Z}_n\times \mathbb{Z}_m$. Moreover, if $i\in V(K_n)$, then by $^i\,K_m$ we denote the subgraph of $K_n\cp K_m$ induced by the vertices $\{i\}\times \mathbb{Z}_m$, and call it a (vertical) \textit{layer}. Symmetrically, for $j\in V(K_m)$, the (horizontal) layer is the subgraph induced by $\mathbb{Z}_n\times \{j\}$, and denoted $K_n^j$.

In order to complete this preliminary section, we determine the suitable values of $k$ for which $\wdim_k(G)$ can be computed, when $G$ is a Hamming graph.

\begin{theorem}
\label{thm:hamming}
If $r\ge 2$ and $n_1 \ge n_2 \ge \cdots \ge n_r \ge 2$, then
$$\kappa(K_{n_1}\cp K_{n_2}\cp \cdots \cp K_{n_r}) = 2n_2\cdots n_r\,.$$
In particular, if $n_1=2$, then $\kappa(Q_r) = 2^r$.
\end{theorem}

\begin{proof}
Let $n_1 \ge \cdots \ge n_r \ge 2$, where $r\ge 2$. Set $G = K_{n_1}\cp \cdots \cp K_{n_r}$ for the rest of the proof. Throughout the proof we will use the fact the the distance between two vertices of $G$ is equal to the number of coordinates in which they differ. Let $x=(x_1, \ldots, x_r)$ and $y=(y_1, \ldots, y_r)$ be arbitrary, different vertices of $G$. We consider the following cases.

Assume first that $d_G(x,y) = 1$. Let $j\in [r]$ be the unique index for which we have $x_j\ne y_j$. If $u=(u_1, \ldots, u_r)$ is a vertex of $G$ with $u_j \in [r]\setminus \{x_j,y_j\}$, then $d_G(x,u) = d_G(y,u)$ and so $\Delta_u(x,y) = 0$. Assume next that $u_j = x_j$ $(\ne y_j)$. Then $d_G(y,u) = d_G(x,u) + 1$ and therefore $\Delta_u(x,y) = 1$. There are $\pi_j = \prod_{ \overset{i=1}{i\ne j}}^r n_i$ vertices $u$ with $u_j = x_j$, where $x$ is one among them. There are the same number of vertices $u$ with $u_j = y_j$ $(\ne x_j)$, and these vertices also contribute $\pi_j$ to $\Delta(x,y)$. It follows that $\Delta(x,y) = 2\pi_j$. Since $n_1 \ge \cdots \ge n_r$, we get
$$
\min \{\Delta(x,y):\ x,y\in V(G), x\ne y\} \le
\min \{\Delta(x,y):\ xy\in E(G)\} = 2\pi_1 = 2n_2\cdots n_r\,.$$

Assume now that $d_G(x,y) \ge 2$. Let $j$ be an arbitrary coordinate such that $x_{j}\ne y_{j}$. Then, as above, each vertex $u$ with $u_{j}=x_{j}$ contributes $1$ to $\Delta(x,y)$, and the same holds for each vertex $u$ with $u_{j}=y_{j}$. Hence
$$\Delta(x,y) \ge 2 \prod_{ \overset{i=1}{i\ne j}}^r n_i \ge 2\prod_{i=2} ^r n_i\,.$$
Proposition~\ref{prop:kappa=min-Delta} completes the argument for the formula.

The particular case of hypercubes follows since $Q_1 \cong K_2$ and $\kappa(K_2) = 2$, and since $Q_r$ is isomorphic to the Cartesian product of $r$ copies of $K_2$.
\end{proof}

%%%%%%%%%%%%%%%%%%%%%%%%%%%%%%%%%%%%%%
\section{Proof of Theorem \ref{thm:wdim-Kn-Kn-all}}
\label{sec:computing-wdim}
%%%%%%%%%%%%%%%%%%%%%%%%%%%%%%%%%%%%%%

We remark that $\kappa(K_n\cp K_n)=2n$, by Theorem \ref{thm:hamming}, and so, we next proceed to compute each of the values of $\wdim_{k}(K_n\cp K_n)$ for any $n\ge 3$. Through the proof, we assume $n\ge 3$ and set $G = K_n\cp K_n$. We will split the argument into several cases separated into subsections.

\subsection{The case $4\le k\le 2n$}\label{subsc:gen-case}

We recall that we are going to prove that
\begin{equation*}
 \wdim_{k}(K_n\cp K_n) =
\begin{cases}
    n\left\lceil\frac{k}{2}\right\rceil; & \mbox{if $k\ge 4$ is even}\,, \\[0.2cm]
    n\left\lceil\frac{k}{2}\right\rceil-1; & \mbox{if $k\ge 5$ is odd}\,.
\end{cases}
\end{equation*}

First, the assertion $\wdim_{2n}(G) = n^2$ can be readily observed by considering, for instance, the vertices $(0,0)$ and $(0,1)$, because only the vertices from the layers $K_n^0$ and $K_n^1$ can contribute to $\Delta((0,0), (0,1))$. Since such vertices contribute exactly $1$, it follows that a weak $(2n)$-resolving set of $G$ must contain all the vertices of $K_n^0$ and $K_n^1$, and consequently all the vertices of $G$. Hence, in the rest we restrict our attention to the cases when $4\le k\le 2n-1$.

For $i\in \mathbb{Z}_n$ set
$$D_i = \{(i,0), (i+1,1), \ldots, (i+n-1,n-1)\}\,,$$
where the computations are done modulo $n$. Intuitively, the $D_i$s are the diagonals of $G$.

\medskip\noindent
{\bf Case 1}: $k = 2n - 2t$, for some $1\le t\le n-2$. \\
Notice that in such situation, $4\le k\le 2n-2$ (an even integer). We claim that the set
$$X_t = \bigcup_{i=t}^{n-1}D_i$$
is a weak $k$-resolving set. See Fig.~\ref{fig:set-X-t} for some fairly representative examples.

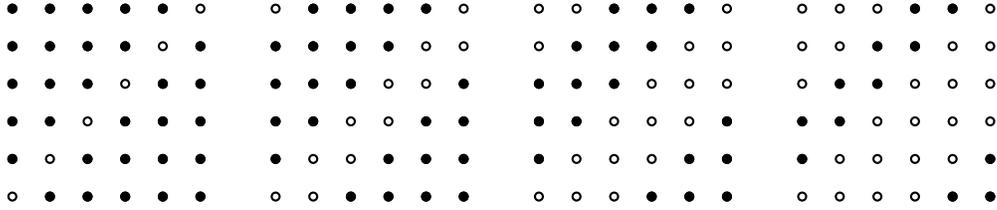
\begin{figure}[ht!]
\begin{center}
\begin{tikzpicture}[scale=0.5,style=thick,x=1cm,y=1cm]
\def\vr{3pt}
\begin{scope}[xshift=0cm, yshift=0cm] % line 0
\coordinate(x0) at (0.0,0.0);
\coordinate(x1) at (1,0);
\coordinate(x2) at (2,0);
\coordinate(x3) at (3,0);
\coordinate(x4) at (4,0);
\coordinate(x5) at (5,0);
\coordinate(y0) at (7,0.0);
\coordinate(y1) at (8,0);
\coordinate(y2) at (9,0);
\coordinate(y3) at (10,0);
\coordinate(y4) at (11,0);
\coordinate(y5) at (12,0);
\coordinate(z0) at (14,0.0);
\coordinate(z1) at (15,0);
\coordinate(z2) at (16,0);
\coordinate(z3) at (17,0);
\coordinate(z4) at (18,0);
\coordinate(z5) at (19,0);
\coordinate(t0) at (21,0.0);
\coordinate(t1) at (22,0);
\coordinate(t2) at (23,0);
\coordinate(t3) at (24,0);
\coordinate(t4) at (25,0);
\coordinate(t5) at (26,0);
% \edges		
%\draw (x5) -- (x6);
%  vertices
\foreach \i in {0,...,5}
{
\draw(x\i)[fill=white] circle(\vr);
}
\foreach \i in {0,...,5}
{
\draw(y\i)[fill=white] circle(\vr);
}
\foreach \i in {0,...,5}
{
\draw(z\i)[fill=white] circle(\vr);
}
\foreach \i in {0,...,5}
{
\draw(t\i)[fill=white] circle(\vr);
}
\foreach \i in {1,...,5}
{
\draw(x\i)[fill=black] circle(\vr);
}
\foreach \i in {2,...,5}
{
\draw(y\i)[fill=black] circle(\vr);
}
\foreach \i in {3,...,5}
{
\draw(z\i)[fill=black] circle(\vr);
}
\foreach \i in {4,...,5}
{
\draw(t\i)[fill=black] circle(\vr);
}
%\draw(x0)[fill=black] circle(\vr);
% text
%\node at (1,-1) {$G$};
\end{scope}
%%%
\begin{scope}[xshift=0cm, yshift=1cm] % line 1
\coordinate(x0) at (0.0,0.0);
\coordinate(x1) at (1,0);
\coordinate(x2) at (2,0);
\coordinate(x3) at (3,0);
\coordinate(x4) at (4,0);
\coordinate(x5) at (5,0);
\coordinate(y0) at (7,0.0);
\coordinate(y1) at (8,0);
\coordinate(y2) at (9,0);
\coordinate(y3) at (10,0);
\coordinate(y4) at (11,0);
\coordinate(y5) at (12,0);
\coordinate(z0) at (14,0.0);
\coordinate(z1) at (15,0);
\coordinate(z2) at (16,0);
\coordinate(z3) at (17,0);
\coordinate(z4) at (18,0);
\coordinate(z5) at (19,0);
\coordinate(t0) at (21,0.0);
\coordinate(t1) at (22,0);
\coordinate(t2) at (23,0);
\coordinate(t3) at (24,0);
\coordinate(t4) at (25,0);
\coordinate(t5) at (26,0);
% \edges		
%\draw (x5) -- (x6);
%  vertices
\foreach \i in {0,...,5}
{
\draw(x\i)[fill=white] circle(\vr);
}
\foreach \i in {0,...,5}
{
\draw(y\i)[fill=white] circle(\vr);
}
\foreach \i in {0,...,5}
{
\draw(z\i)[fill=white] circle(\vr);
}
\foreach \i in {0,...,5}
{
\draw(t\i)[fill=white] circle(\vr);
}
%%%
\foreach \i in {0,2,3,4,5}
{
\draw(x\i)[fill=black] circle(\vr);
}
\foreach \i in {0,3,4,5}
{
\draw(y\i)[fill=black] circle(\vr);
}
\foreach \i in {0,4,5}
{
\draw(z\i)[fill=black] circle(\vr);
}
\foreach \i in {0,5}
{
\draw(t\i)[fill=black] circle(\vr);
}
%\draw(x0)[fill=black] circle(\vr);
% text
%\node at (1,-1) {$G$};
\end{scope}
%%%
\begin{scope}[xshift=0cm, yshift=2cm] % line 2
\coordinate(x0) at (0.0,0.0);
\coordinate(x1) at (1,0);
\coordinate(x2) at (2,0);
\coordinate(x3) at (3,0);
\coordinate(x4) at (4,0);
\coordinate(x5) at (5,0);
\coordinate(y0) at (7,0.0);
\coordinate(y1) at (8,0);
\coordinate(y2) at (9,0);
\coordinate(y3) at (10,0);
\coordinate(y4) at (11,0);
\coordinate(y5) at (12,0);
\coordinate(z0) at (14,0.0);
\coordinate(z1) at (15,0);
\coordinate(z2) at (16,0);
\coordinate(z3) at (17,0);
\coordinate(z4) at (18,0);
\coordinate(z5) at (19,0);
\coordinate(t0) at (21,0.0);
\coordinate(t1) at (22,0);
\coordinate(t2) at (23,0);
\coordinate(t3) at (24,0);
\coordinate(t4) at (25,0);
\coordinate(t5) at (26,0);
% \edges		
%\draw (x5) -- (x6);
%  vertices
\foreach \i in {0,...,5}
{
\draw(x\i)[fill=white] circle(\vr);
}
\foreach \i in {0,...,5}
{
\draw(y\i)[fill=white] circle(\vr);
}
\foreach \i in {0,...,5}
{
\draw(z\i)[fill=white] circle(\vr);
}
\foreach \i in {0,...,5}
{
\draw(t\i)[fill=white] circle(\vr);
}
%%%
\foreach \i in {0,1,3,4,5}
{
\draw(x\i)[fill=black] circle(\vr);
}
\foreach \i in {0,1,4,5}
{
\draw(y\i)[fill=black] circle(\vr);
}
\foreach \i in {0,1,5}
{
\draw(z\i)[fill=black] circle(\vr);
}
\foreach \i in {0,1}
{
\draw(t\i)[fill=black] circle(\vr);
}
%\draw(x0)[fill=black] circle(\vr);
% text
%\node at (1,-1) {$G$};
\end{scope}
%%%
\begin{scope}[xshift=0cm, yshift=3cm] % line 3
\coordinate(x0) at (0.0,0.0);
\coordinate(x1) at (1,0);
\coordinate(x2) at (2,0);
\coordinate(x3) at (3,0);
\coordinate(x4) at (4,0);
\coordinate(x5) at (5,0);
\coordinate(y0) at (7,0.0);
\coordinate(y1) at (8,0);
\coordinate(y2) at (9,0);
\coordinate(y3) at (10,0);
\coordinate(y4) at (11,0);
\coordinate(y5) at (12,0);
\coordinate(z0) at (14,0.0);
\coordinate(z1) at (15,0);
\coordinate(z2) at (16,0);
\coordinate(z3) at (17,0);
\coordinate(z4) at (18,0);
\coordinate(z5) at (19,0);
\coordinate(t0) at (21,0.0);
\coordinate(t1) at (22,0);
\coordinate(t2) at (23,0);
\coordinate(t3) at (24,0);
\coordinate(t4) at (25,0);
\coordinate(t5) at (26,0);
% \edges		
%\draw (x5) -- (x6);
%  vertices
\foreach \i in {0,...,5}
{
\draw(x\i)[fill=white] circle(\vr);
}
\foreach \i in {0,...,5}
{
\draw(y\i)[fill=white] circle(\vr);
}
\foreach \i in {0,...,5}
{
\draw(z\i)[fill=white] circle(\vr);
}
\foreach \i in {0,...,5}
{
\draw(t\i)[fill=white] circle(\vr);
}
%%%%
\foreach \i in {0,1,2,4,5}
{
\draw(x\i)[fill=black] circle(\vr);
}
\foreach \i in {0,1,2,5}
{
\draw(y\i)[fill=black] circle(\vr);
}
\foreach \i in {0,1,2}
{
\draw(z\i)[fill=black] circle(\vr);
}
\foreach \i in {1,2}
{
\draw(t\i)[fill=black] circle(\vr);
}
%\draw(x0)[fill=black] circle(\vr);
% text
%\node at (1,-1) {$G$};
\end{scope}
%%%
\begin{scope}[xshift=0cm, yshift=4cm] % line 4
\coordinate(x0) at (0.0,0.0);
\coordinate(x1) at (1,0);
\coordinate(x2) at (2,0);
\coordinate(x3) at (3,0);
\coordinate(x4) at (4,0);
\coordinate(x5) at (5,0);
\coordinate(y0) at (7,0.0);
\coordinate(y1) at (8,0);
\coordinate(y2) at (9,0);
\coordinate(y3) at (10,0);
\coordinate(y4) at (11,0);
\coordinate(y5) at (12,0);
\coordinate(z0) at (14,0.0);
\coordinate(z1) at (15,0);
\coordinate(z2) at (16,0);
\coordinate(z3) at (17,0);
\coordinate(z4) at (18,0);
\coordinate(z5) at (19,0);
\coordinate(t0) at (21,0.0);
\coordinate(t1) at (22,0);
\coordinate(t2) at (23,0);
\coordinate(t3) at (24,0);
\coordinate(t4) at (25,0);
\coordinate(t5) at (26,0);
% \edges		
%\draw (x5) -- (x6);
%  vertices
\foreach \i in {0,...,5}
{
\draw(x\i)[fill=white] circle(\vr);
}
\foreach \i in {0,...,5}
{
\draw(y\i)[fill=white] circle(\vr);
}
\foreach \i in {0,...,5}
{
\draw(z\i)[fill=white] circle(\vr);
}
\foreach \i in {0,...,5}
{
\draw(t\i)[fill=white] circle(\vr);
}
%%%
\foreach \i in {0,1,2,3,5}
{
\draw(x\i)[fill=black] circle(\vr);
}
\foreach \i in {0,1,2,3}
{
\draw(y\i)[fill=black] circle(\vr);
}
\foreach \i in {1,2,3}
{
\draw(z\i)[fill=black] circle(\vr);
}
\foreach \i in {2,3}
{
\draw(t\i)[fill=black] circle(\vr);
}
%\draw(x0)[fill=black] circle(\vr);
% text
%\node at (1,-1) {$G$};
\end{scope}
%%%%
\begin{scope}[xshift=0cm, yshift=5cm] % line 5
\coordinate(x0) at (0.0,0.0);
\coordinate(x1) at (1,0);
\coordinate(x2) at (2,0);
\coordinate(x3) at (3,0);
\coordinate(x4) at (4,0);
\coordinate(x5) at (5,0);
\coordinate(y0) at (7,0.0);
\coordinate(y1) at (8,0);
\coordinate(y2) at (9,0);
\coordinate(y3) at (10,0);
\coordinate(y4) at (11,0);
\coordinate(y5) at (12,0);
\coordinate(z0) at (14,0.0);
\coordinate(z1) at (15,0);
\coordinate(z2) at (16,0);
\coordinate(z3) at (17,0);
\coordinate(z4) at (18,0);
\coordinate(z5) at (19,0);
\coordinate(t0) at (21,0.0);
\coordinate(t1) at (22,0);
\coordinate(t2) at (23,0);
\coordinate(t3) at (24,0);
\coordinate(t4) at (25,0);
\coordinate(t5) at (26,0);
% \edges		
%\draw (x5) -- (x6);
%  vertices
\foreach \i in {0,...,5}
{
\draw(x\i)[fill=white] circle(\vr);
}
\foreach \i in {0,...,5}
{
\draw(y\i)[fill=white] circle(\vr);
}
\foreach \i in {0,...,5}
{
\draw(z\i)[fill=white] circle(\vr);
}
\foreach \i in {0,...,5}
{
\draw(t\i)[fill=white] circle(\vr);
}
%%%%
\foreach \i in {0,...,4}
{
\draw(x\i)[fill=black] circle(\vr);
}
\foreach \i in {1,...,4}
{
\draw(y\i)[fill=black] circle(\vr);
}
\foreach \i in {2,...,4}
{
\draw(z\i)[fill=black] circle(\vr);
}
\foreach \i in {3,4}
{
\draw(t\i)[fill=black] circle(\vr);
}
%\draw(x0)[fill=black] circle(\vr);
% text
%\node at (1,-1) {$G$};
\end{scope}
\end{tikzpicture}
\caption{The sets (in bold) $X_1$ (a weak $10$-metric basis), $X_2$ (a weak $8$-metric basis), $X_3$ (a weak $6$-metric basis) and $X_4$ (a weak $4$-metric basis), respectively, in $K_6\cp K_6$}
\label{fig:set-X-t}
\end{center}
\end{figure}

For this sake, note first that $X_t$ contains precisely $\frac{k}{2}=n-t$ vertices in each (horizontal and vertical) layer of $G$.
Consider now arbitrary vertices $(i,j)$ and $(i',j')$ of $G$ and distinguish two different situations. If $i=i'$, then in each of the layers $K_n^j$ and $K_n^{j'}$ there are $\frac{k}{2}=n-t$ vertices (where $(i,j)$ and $(i,j')$ could belong to them) that contribute $1$ to $\Delta_{X_t}((i,j), (i,j'))$, so that $\Delta_{X_t}((i,j)(i,j')) \ge 2n-2t=k$ as required. The situation when $j=j'$ is symmetric.
Assume now that $i\ne i'$ and $j\ne j'$. Then the layers $^iK_n$, $^{i'}K_n$, $K_n^j$, and $K_n^{j'}$ are pairwise different layers that intersect in the vertices $(i,j)$, $(i,j')$, $(i',j)$, and $(i',j')$. Each of these four vertices might not contribute to $\Delta_{X_t}((i,j), (i',j'))$. Since each layer contains $\frac{k}{2}=n-t$ vertices of $X_t$, it thus follows that
$$\Delta_{X_t}((i,j), (i',j')) \ge 4\frac{k}{4} - 4 = 2k - 4 \ge k\,,$$
where the last inequality holds since $k\ge 4$. Consequently, $X_t$ is a weak $k$-resolving set as claimed. Hence $\wdim_{k}(G) \le n\frac{k}{2}=n\left\lceil\frac{k}{2}\right\rceil$.

\medskip
To prove that $\wdim_{k}(G) \ge n\left\lceil\frac{k}{2}\right\rceil$, suppose on the contrary that there exists a weak $k$-resolving set $Y$ of $G$ with $|Y|\le n\left\lceil\frac{k}{2}\right\rceil-1$. By the pigeonhole principle there exists a layer, we may assume without loss of generality to be $K_n^0$, such that $x = |Y\cap V(K_n^0)| \le \frac{k}{2}-1$.
Consider now the vertices $(0,0)$ and $(0,j)$ with $j\ne 0$. Let $y = |Y\cap V(K_n^j)|$.
Note that  $x + y = \Delta_{Y}((0,0), (0,j)) \ge k$. Since $x\le \frac{k}{2}-1$, this in turn implies that $y\ge \frac{k}{2} + 1$. As this holds for any $j\ne 0$, we consequently have $|Y| \ge x + (n-1)(\frac{k}{2}+1)$ and so, by using our assumption on the cardinality of $Y$,
$$n\frac{k}{2} > |Y| \ge x + (n-1)\left(\frac{k}{2}+1\right)\,,$$
which implies that $x < 0$, since $k \le 2n-2$, and this is not possible. This contradiction proves that $\wdim_{k}(G) \ge n\left\lceil\frac{k}{2}\right\rceil$ and thus the equality follows in the case $k$ is even.

\medskip\noindent
{\bf Case 2}: $k = 2n - 2t-1$, for some $0\le t\le n-3$. \\
Notice that in such situation, $5\le k\le 2n-1$ (an odd integer). To see that $\wdim_{k}(G) \le n\left\lceil\frac{k}{2}\right\rceil-1$, we claim that the set $X_t'$ obtained from $X_t$ by removing the vertex $(1,2)$ is a weak $k$-resolving set, i,e., $X_t'=X_t\setminus\{(1,2)\}$. Let $(i,j),(i',j')\in V(G)$ be any two arbitrary vertices. If $(1,2)\notin \{(i,j),(i',j')\}$, then we can use the argument of Case 1, that is, in its proof, while considering $\Delta_{X_t}((i,j), (i',j'))$, we have only considered contributions of $1$ of each vertex from $X_t$. Hence, by using the same arguments, deleting only one vertex from $X_t$ yields a set such that any two vertices from $G$ the new set contributes at least $\Delta_{X_t}((i,j), (i',j')) - 1$ to   $\Delta_{X_t'}((i,j), (i',j'))$.

Assume now that (WLOG) $(i',j')=(1,2)$. First notice that, since $X'_t$ is obtained from $X_t$ by removing one vertex, it holds that each layer of $G$ (vertical or horizontal)  contains $(k+1)/2$ vertices of $X'_t$, with the exception of that layers containing the vertex $(1,2)$. We have now two different situations.

\medskip\noindent
{\bf Case 2.1}: $i\ne 1$ and $j\ne 2$.\\
First, if $(i,j)\in X'_t$, then $(i,j)$ contributes with $2$ to $\Delta_{X'_t}((i,j), (i',j'))$. In addition, based on the fact that there are at least $(k+1)/2-2$ vertices in each of the four layers that contribute with $1$ to $\Delta_{X'_t}((i,j), (i',j'))$, and by the fact that $k\ge 5$, we deduce that
$$\Delta_{X'_t}((i,j), (i',j')) \ge 4\left(\frac{k+1}{2}-2\right)+2=2k-4\ge k+1.$$
On the other hand, if $(i,j)\notin X'_t$, then $(i,j)$ clearly contributes nothing to $\Delta_{X'_t}((i,j), (i',j'))$. However, now there are two layers (the ones containing $(i,j)$) such that they contain $(k+1)/2-1$ vertices in each of the two such layers that contribute with $1$ to $\Delta_{X'_t}((i,j), (i',j'))$. In addition, in the other two layers (the ones containing $(1,2)$) there are $(k+1)/2-2$ vertices in each of them that contribute with $1$ to $\Delta_{X'_t}((i,j), (i',j'))$. Hence, having again in mind that $k\ge 5$, it holds that
$$\Delta_{X'_t}((i,j), (i',j')) \ge 2\left(\frac{k+1}{2}-1\right)+2\left(\frac{k+1}{2}-2\right)=2k-5\ge k.$$

\medskip\noindent
{\bf Case 2.2}: $i=1$ or $j=2$. \\
By the symmetry of $G$, it suffices to consider that $i=1$.
Notice that there are $(k+1)/2$ vertices from $X'_t$ in the layer $K_n^j$ that contribute with $1$ to $\Delta_{X'_t}((i,j), (i',j'))$, as well as, there are $(k+1)/2-1$ vertices from $X'_t$ in the layer $K_n^2$ that contribute with $1$ to $\Delta_{X'_t}((i,j), (i',j'))$. Thus,
$$\Delta_{X'_t}((i,j)(i',j')) \ge \left(\frac{k+1}{2}\right)+\left(\frac{k+1}{2}-1\right)= k.$$

As a consequence of the arguments above, we obtain that $X'_t$ is a weak $k$-resolving set of $G$, and so, $\wdim_{k}(K_n\cp K_n)\le n\left\lceil\frac{k}{2}\right\rceil-1$ if $k\ge 5$ is odd.

\medskip
It remains to see that $\wdim_{k}(G) \ge n\left\lceil\frac{k}{2}\right\rceil-1$. Suppose on the contrary that there exists a weak $k$-resolving set $Y$ of $G$ with $|Y|\le n\left\lceil\frac{k}{2}\right\rceil-2$. By the pigeonhole principle one of the following two possibilities occur.

Assume first that there is a layer $K_n^\ell$ such that $|Y\cap V(K_n^\ell)| \le \left\lceil\frac{k}{2}\right\rceil -2$.
We may consider (WLOG) that $\ell = 0$ and let $x = |Y\cap V(K_n^0)| \le \left\lceil\frac{k}{2}\right\rceil-2$.
Consider now the vertices $(0,0)$ and $(0,j)$ with $j\ne 0$. Let $y = |Y\cap V(K_n^j)|$.
Since $x + y = \Delta_{Y}((0,0), (0,j)) \ge 2\left\lceil\frac{k}{2}\right\rceil-1$, and because $x\le \left\lceil\frac{k}{2}\right\rceil-2$, we get $y\ge \left\lceil\frac{k}{2}\right\rceil + 1$. As this holds for any $j\ne 0$, we consequently have $|Y| \ge x + (n-1)(\left\lceil\frac{k}{2}\right\rceil+1)$, and so, by using our assumption on the cardinality of $Y$ we deduce
$$n\left\lceil\frac{k}{2}\right\rceil -1 > |Y| \ge x + (n-1)\left(\left\lceil\frac{k}{2}\right\rceil+1\right),$$
which implies that $x < 0$, a contradiction.

On the other hand, assume now there exist $\ell, \ell'\in \mathbb{Z}_n$, $\ell\ne \ell'$ such that $|Y\cap V(K_n^\ell)| \le \left\lceil\frac{k}{2}\right\rceil -1$ and $|Y\cap V(K_n^{\ell'})| \le \left\lceil\frac{k}{2}\right\rceil -1$.
Consider the vertices $(0,\ell)$ and $(0,\ell')$. Then
$$\Delta_{Y}((0,\ell), (0,\ell')) = |Y\cap V(K_n^\ell)| + |Y\cap V(K_n^{\ell'})| \le 2\left(\left\lceil\frac{k}{2}\right\rceil-1\right) = 2\left\lceil\frac{k}{2}\right\rceil -2\,,$$
which is not possible.

\medskip
The above considerations demonstrate that  $|Y|\le n\left\lceil\frac{k}{2}\right\rceil-2$ is not possible, hence   $\wdim_{k}(G) \ge n\left\lceil\frac{k}{2}\right\rceil-1$. We can conclude that $\wdim_{k}(G) = \left\lceil\frac{k}{2}\right\rceil-1$, when $5\le k\le 2n-1$ is odd.

\subsection{The case $k=3$}

We recall that in this subsection, we are going to show that
$$\wdim_3(K_n\cp K_n)=2n.$$
First, since $\wdim_4(G)=2n$ as proved in Subsection \ref{subsc:gen-case}, and because $\wdim_3(G)\le \wdim_4(G)$, we have that $\wdim_3(G)\le 2n$.

\medskip
Suppose now that $\wdim_3(G)\le 2n-1$, and let $Y$ be a weak $3$-metric basis. We shall show some structural properties of $Y$.

\smallskip
\noindent
\textbf{Claim:} Every layer of $G$ contains at least one vertex of $Y$.

\smallskip
\noindent
Indeed, suppose (WLOG) that $V(K_n^0)\cap Y=\emptyset$. Now, since for any vertex $(i,0)$, $i\ne 0$, it must hold that $\Delta_{Y}((0,0), (i,0))\ge 3$, we deduce that  $|V(K_n^i)\cap Y|\ge 3$. Thus, since this happens for each $i\ne 0$, we obtain that $|Y|\ge 3(n-1)>2n-1$ because $n\ge 3$.

\smallskip
We remark that this claim is satisfied for each layer, no matter if it is vertical or horizontal. Since $|Y|\le 2n -1$ by assumption, we have at least one horizontal layer $K_n^j$ with $|V(K_n^j)\cap Y|=1$ and at least one vertical layer $^iK_n$ with $|V(^iK_n)\cap Y|=1$.

Consider a vertex $(i,j)\notin Y$ such that $|V(^iK_n)\cap Y|=1$. If there are two horizontal layers $K_n^{j'}$ and $K_n^{j''}$, each having at most one vertex from $Y$, then the vertices $\Delta_Y((i,j'),(i,j''))\le 2$, which is not possible. This, together with the fact that $|Y|\le 2n-1$, and also with the claim above, imply that all but one horizontal layer have exactly two vertices of $Y$. Moreover, such remaining horizontal layer has exactly one vertex of $Y$. Let $K_n^{j^*}$ be such a layer. In addition, by a symmetric argument we obtain a parallel conclusion for vertical layers where, by our assumption, the layer $^iK_n$ is the one that has exactly one vertex of $Y$.  We have the following cases.

\medskip
\noindent
\textbf{Case 1:} $j^*=j$.\\
Let $(x,y)\in Y$ such that $x\ne i$ and $y\ne j$. We consider the vertices $(x,j),(i,y)$. Then, since $(i,j)$ and $(x,y)$ do not contribute to $\Delta_Y((x,j),(i,y))$, there must be at least three vertices of $Y$ lying in $V(^x K_n)\cup V(K_n^y)\setminus \{(x,y)\}$. This means that $^x K_n$ or $K_n^y$ contains at least three vertices of $Y$ (including the vertex $(x,y)$), which is a contradiction.

\medskip
\noindent
\textbf{Case 2:} $j^*\ne j$.\\
Let $(i^*,j^*)$ be the unique vertex of $Y$ in $K_n^{j^*}$. If $(i^*,j)\notin Y$, then by considering the vertices $(i^*,j)$ and $(i,j^*)$ and using a parallel argument as in the Case 1 above, we find one layer with at least three vertices of $Y$, which is again a contradiction.

Hence, we suppose that $(i^*,j)\in Y$. Consider then the vertices $(i^*,j)$ and $(i,j^*)$. Note that $(i^*,j^*)$ and $(i,j)$ do not contribute to $\Delta_Y((i^*,j),(i,j^*))$, and that $(i^*,j)$ contributes with $2$. Thus, there must be an additional vertex from $Y$ in the four layers containing $(i^*,j)$ and $(i,j^*)$. That is, in the set
$$(V(^{i^*}K_n)\cup V(^iK_n)\cup V(K_n^j)\cup V(K_n^{j^*}))\setminus \{(i^*,j),(i,j),(i^*,j^*)\}.$$ However, such a vertex from $Y$ does not exist by our assumptions, that are:
\begin{itemize}
    \item $Y\cap V(^iK_n)=\{(i,j)\}$,
    \item $Y\cap V(K_n^{j^*})=\{(i^*,j^*)\}$,
    \item $Y\cap V(^{i^*}K_n)=\{(i^*,j),(i^*,j^*)\}$, and
    \item $Y\cap V(K_n^j)=\{(i^*,j),(i,j)\}$.
\end{itemize}
This is a final contradiction, that shows that a weak $3$-resolving set of $G$ contains at least $2n$ vertices. This settles the case $k=3$.

\subsection{The case $k=2$}

Recall that in this subsection, our aim is to prove that
$$\wdim_2(K_n\cp K_n) = \left\lceil \frac{4n}{3}\right\rceil.$$

To this end, the following auxiliary graph will be useful. Let $Y\subseteq V(G)$. Then we define the graph $G_Y$ as follows. $G_Y$ is a bipartite graph with a bipartition $V_1 = \{0,1,\ldots, n-1\}$, $V_2 = \{0',1', \ldots, (n-1)'\}$, and the vertex $i\in V_1$ is adjacent to the vertex $j'\in V_2$ if $(i,j)\in Y$.

If $n\in \{3,4,5\}$, then we have checked the assertion of the theorem by using a computer, but it can also be checked by hand. Hence, assume in the following that $n\ge 6$.

We begin by constructing special weak $2$-resolving sets $Y_n$ of $G$ as follows. Let $n =3s + t$, $s \ge 2$, $t\in \mathbb{Z}_3$, and distinguish the following cases.
\begin{itemize}
\item If $t = 0$, then $Y_n = \{(3r,3r), (3r,3r+1), (3r+1,3r+2), (3r+2, 3r+2):\ r\in \mathbb{Z}_s\}$.
\item If $t = 1$, then $Y_n = Y_{n-1} \cup \{(n-1,n-2), (n-1,n-1)\}$.
\item If $t = 2$, then $Y_n = Y_{n-2} \cup \{(n-2,n-2), (n-2,n-1), (n-1,n-1)\}$.
\end{itemize}
See Fig.~\ref{fig:6-7-8} where the sets $Y_6$, $Y_7$, and $Y_8$ are schematically shown in  $K_6\cp K_6$, $K_7\cp K_7$, and $K_8\cp K_8$, respectively, and the edges are not drawn to make the construction clearer.

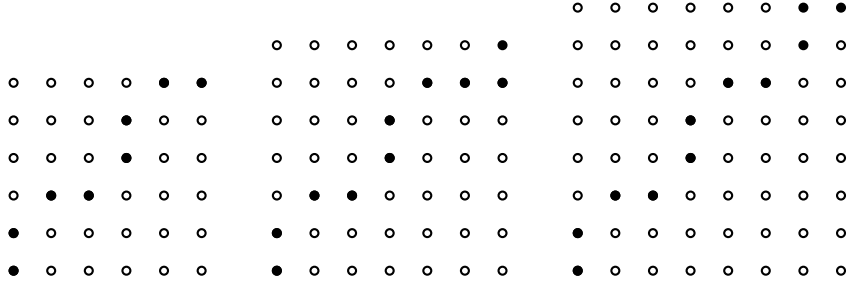
\begin{figure}[ht!]
\begin{center}
\begin{tikzpicture}[scale=0.5,style=thick,x=1cm,y=1cm]
\def\vr{3pt}
\begin{scope}[xshift=0cm, yshift=0cm] % line 0
\coordinate(x0) at (0.0,0.0);
\coordinate(x1) at (1,0);
\coordinate(x2) at (2,0);
\coordinate(x3) at (3,0);
\coordinate(x4) at (4,0);
\coordinate(x5) at (5,0);
\coordinate(y0) at (7,0.0);
\coordinate(y1) at (8,0);
\coordinate(y2) at (9,0);
\coordinate(y3) at (10,0);
\coordinate(y4) at (11,0);
\coordinate(y5) at (12,0);
\coordinate(y6) at (13,0);
\coordinate(z0) at (15,0.0);
\coordinate(z1) at (16,0);
\coordinate(z2) at (17,0);
\coordinate(z3) at (18,0);
\coordinate(z4) at (19,0);
\coordinate(z5) at (20,0);
\coordinate(z6) at (21,0);
\coordinate(z7) at (22,0);
% \edges		
%\draw (x5) -- (x6);
%  vertices
\foreach \i in {0,...,5}
{
\draw(x\i)[fill=white] circle(\vr);
}
\foreach \i in {0,...,6}
{
\draw(y\i)[fill=white] circle(\vr);
}
\foreach \i in {0,...,7}
{
\draw(z\i)[fill=white] circle(\vr);
}

\draw(x0)[fill=black] circle(\vr);
\draw(y0)[fill=black] circle(\vr);
\draw(z0)[fill=black] circle(\vr);
% text
%\node at (1,-1) {$G$};
\end{scope}
\begin{scope}[xshift=0cm, yshift=1cm] % line 1
\coordinate(x0) at (0.0,0.0);
\coordinate(x1) at (1,0);
\coordinate(x2) at (2,0);
\coordinate(x3) at (3,0);
\coordinate(x4) at (4,0);
\coordinate(x5) at (5,0);
\coordinate(y0) at (7,0.0);
\coordinate(y1) at (8,0);
\coordinate(y2) at (9,0);
\coordinate(y3) at (10,0);
\coordinate(y4) at (11,0);
\coordinate(y5) at (12,0);
\coordinate(y6) at (13,0);
\coordinate(z0) at (15,0.0);
\coordinate(z1) at (16,0);
\coordinate(z2) at (17,0);
\coordinate(z3) at (18,0);
\coordinate(z4) at (19,0);
\coordinate(z5) at (20,0);
\coordinate(z6) at (21,0);
\coordinate(z7) at (22,0);
% \edges		
%\draw (x5) -- (x6);
%  vertices
\foreach \i in {0,...,5}
{
\draw(x\i)[fill=white] circle(\vr);
}
\foreach \i in {0,...,6}
{
\draw(y\i)[fill=white] circle(\vr);
}
\foreach \i in {0,...,7}
{
\draw(z\i)[fill=white] circle(\vr);
}

\draw(x0)[fill=black] circle(\vr);
\draw(y0)[fill=black] circle(\vr);
\draw(z0)[fill=black] circle(\vr);
% text
%\node at (1,-1) {$G$};
\end{scope}
\begin{scope}[xshift=0cm, yshift=2cm] % line 2
\coordinate(x0) at (0.0,0.0);
\coordinate(x1) at (1,0);
\coordinate(x2) at (2,0);
\coordinate(x3) at (3,0);
\coordinate(x4) at (4,0);
\coordinate(x5) at (5,0);
\coordinate(y0) at (7,0.0);
\coordinate(y1) at (8,0);
\coordinate(y2) at (9,0);
\coordinate(y3) at (10,0);
\coordinate(y4) at (11,0);
\coordinate(y5) at (12,0);
\coordinate(y6) at (13,0);
\coordinate(z0) at (15,0.0);
\coordinate(z1) at (16,0);
\coordinate(z2) at (17,0);
\coordinate(z3) at (18,0);
\coordinate(z4) at (19,0);
\coordinate(z5) at (20,0);
\coordinate(z6) at (21,0);
\coordinate(z7) at (22,0);
% \edges		
%\draw (x5) -- (x6);
%  vertices
\foreach \i in {0,...,5}
{
\draw(x\i)[fill=white] circle(\vr);
}
\foreach \i in {0,...,6}
{
\draw(y\i)[fill=white] circle(\vr);
}
\foreach \i in {0,...,7}
{
\draw(z\i)[fill=white] circle(\vr);
}

\draw(x1)[fill=black] circle(\vr);
\draw(y1)[fill=black] circle(\vr);
\draw(z1)[fill=black] circle(\vr);
\draw(x2)[fill=black] circle(\vr);
\draw(y2)[fill=black] circle(\vr);
\draw(z2)[fill=black] circle(\vr);

% text
%\node at (1,-1) {$G$};
\end{scope}
\begin{scope}[xshift=0cm, yshift=3cm] % line 3
\coordinate(x0) at (0.0,0.0);
\coordinate(x1) at (1,0);
\coordinate(x2) at (2,0);
\coordinate(x3) at (3,0);
\coordinate(x4) at (4,0);
\coordinate(x5) at (5,0);
\coordinate(y0) at (7,0.0);
\coordinate(y1) at (8,0);
\coordinate(y2) at (9,0);
\coordinate(y3) at (10,0);
\coordinate(y4) at (11,0);
\coordinate(y5) at (12,0);
\coordinate(y6) at (13,0);
\coordinate(z0) at (15,0.0);
\coordinate(z1) at (16,0);
\coordinate(z2) at (17,0);
\coordinate(z3) at (18,0);
\coordinate(z4) at (19,0);
\coordinate(z5) at (20,0);
\coordinate(z6) at (21,0);
\coordinate(z7) at (22,0);
% \edges		
%\draw (x5) -- (x6);
%  vertices
\foreach \i in {0,...,5}
{
\draw(x\i)[fill=white] circle(\vr);
}
\foreach \i in {0,...,6}
{
\draw(y\i)[fill=white] circle(\vr);
}
\foreach \i in {0,...,7}
{
\draw(z\i)[fill=white] circle(\vr);
}

\draw(x3)[fill=black] circle(\vr);
\draw(y3)[fill=black] circle(\vr);
\draw(z3)[fill=black] circle(\vr);
% text
%\node at (1,-1) {$G$};
\end{scope}
\begin{scope}[xshift=0cm, yshift=4cm] % line 4
\coordinate(x0) at (0.0,0.0);
\coordinate(x1) at (1,0);
\coordinate(x2) at (2,0);
\coordinate(x3) at (3,0);
\coordinate(x4) at (4,0);
\coordinate(x5) at (5,0);
\coordinate(y0) at (7,0.0);
\coordinate(y1) at (8,0);
\coordinate(y2) at (9,0);
\coordinate(y3) at (10,0);
\coordinate(y4) at (11,0);
\coordinate(y5) at (12,0);
\coordinate(y6) at (13,0);
\coordinate(z0) at (15,0.0);
\coordinate(z1) at (16,0);
\coordinate(z2) at (17,0);
\coordinate(z3) at (18,0);
\coordinate(z4) at (19,0);
\coordinate(z5) at (20,0);
\coordinate(z6) at (21,0);
\coordinate(z7) at (22,0);
% \edges		
%\draw (x5) -- (x6);
%  vertices
\foreach \i in {0,...,5}
{
\draw(x\i)[fill=white] circle(\vr);
}
\foreach \i in {0,...,6}
{
\draw(y\i)[fill=white] circle(\vr);
}
\foreach \i in {0,...,7}
{
\draw(z\i)[fill=white] circle(\vr);
}
\draw(x3)[fill=black] circle(\vr);
\draw(y3)[fill=black] circle(\vr);
\draw(z3)[fill=black] circle(\vr);
% text
%\node at (1,-1) {$G$};
\end{scope}
\begin{scope}[xshift=0cm, yshift=5cm] % line 5
\coordinate(x0) at (0.0,0.0);
\coordinate(x1) at (1,0);
\coordinate(x2) at (2,0);
\coordinate(x3) at (3,0);
\coordinate(x4) at (4,0);
\coordinate(x5) at (5,0);
\coordinate(y0) at (7,0.0);
\coordinate(y1) at (8,0);
\coordinate(y2) at (9,0);
\coordinate(y3) at (10,0);
\coordinate(y4) at (11,0);
\coordinate(y5) at (12,0);
\coordinate(y6) at (13,0);
\coordinate(z0) at (15,0.0);
\coordinate(z1) at (16,0);
\coordinate(z2) at (17,0);
\coordinate(z3) at (18,0);
\coordinate(z4) at (19,0);
\coordinate(z5) at (20,0);
\coordinate(z6) at (21,0);
\coordinate(z7) at (22,0);
% \edges		
%\draw (x5) -- (x6);
%  vertices
\foreach \i in {0,...,5}
{
\draw(x\i)[fill=white] circle(\vr);
}
\foreach \i in {0,...,6}
{
\draw(y\i)[fill=white] circle(\vr);
}
\foreach \i in {0,...,7}
{
\draw(z\i)[fill=white] circle(\vr);
}

\draw(x4)[fill=black] circle(\vr);
\draw(y4)[fill=black] circle(\vr);
\draw(z4)[fill=black] circle(\vr);
\draw(x5)[fill=black] circle(\vr);
\draw(y5)[fill=black] circle(\vr);
\draw(z5)[fill=black] circle(\vr);

% text
%\node at (1,-1) {$G$};
\end{scope}
%
% extra line in n=7
\draw(7,6)[fill=white] circle(\vr);
\draw(8,6)[fill=white] circle(\vr);
\draw(9,6)[fill=white] circle(\vr);
\draw(10,6)[fill=white] circle(\vr);
\draw(11,6)[fill=white] circle(\vr);
\draw(12,6)[fill=white] circle(\vr);
\draw(13,6)[fill=black] circle(\vr);
\draw(13,5)[fill=black] circle(\vr);
% extra lines in n=8
\draw(15,6)[fill=white] circle(\vr);
\draw(16,6)[fill=white] circle(\vr);
\draw(17,6)[fill=white] circle(\vr);
\draw(18,6)[fill=white] circle(\vr);
\draw(19,6)[fill=white] circle(\vr);
\draw(20,6)[fill=white] circle(\vr);
\draw(21,6)[fill=black] circle(\vr);
\draw(22,6)[fill=white] circle(\vr);
\draw(15,7)[fill=white] circle(\vr);
\draw(16,7)[fill=white] circle(\vr);
\draw(17,7)[fill=white] circle(\vr);
\draw(18,7)[fill=white] circle(\vr);
\draw(19,7)[fill=white] circle(\vr);
\draw(20,7)[fill=white] circle(\vr);
\draw(21,7)[fill=black] circle(\vr);
\draw(22,7)[fill=black] circle(\vr);
\end{tikzpicture}
\caption{The sets $Y_6$, $Y_7$, and $Y_8$ respectively in $K_6\cp K_6$, $K_7\cp K_7$, and $K_8\cp K_8$}
\label{fig:6-7-8}
\end{center}
\end{figure}

We claim that $Y_n$ is a weak $2$-resolving set of $G$.
To this end, consider first two vertices with the same first coordinate, say $(i,j)$ and $(i,j')$.
Then each of the layers $K_n^j$ and $K_n^{j'}$ contains at least one vertex of $Y_n$ which already implies that $\Delta_{Y_n}((i,j), (i,j')) \ge 2$.
Analogously, we see that $\Delta_{Y_n}((i,j), (i',j)) \ge 2$ for any $j$ and any $i\ne i'$. Consider next vertices $(i,j)$ and $(i',j')$, where $i\ne i'$ and $j\ne j'$. Setting
$$Y_n(i,i',j,j') = Y_n \cap (V(K_n^{j}) \cup V(K_n^{j'}) \cup V(^{i}K_n) \cup V(^{i'}K_n))$$
we infer that $|Y_n(i,i',j,j')| \ge 4$. Since each vertex from
$$Y_n(i,i',j,j')\setminus \{(i,j'), (i',j)\}$$
contributes to $\Delta_{Y_n}((i,j), (i',j'))$, we infer that also now we have $\Delta_{Y_n}((i,j), (i',j')) \ge 2$.

We have thus proved that $Y_n$ is a weak $2$-resolving set of $G$. As $|Y_n| = \left\lceil \frac{4n}{3}\right\rceil$, we conclude that $\wdim_2(K_n\cp K_n) \le \left\lceil \frac{4n}{3}\right\rceil$.

\medskip
To complete the proof we need to demonstrate that $\wdim_2(K_n\cp K_n) \ge \left\lceil \frac{4n}{3}\right\rceil$. For this sake let $Y$ be an arbitrary weak $2$-metric basis of $G$ and consider the associated graph $G_Y$.

We first claim that $G_Y$ has no isolated vertices. Suppose on the contrary that, without loss of generality, the vertex $0\in V_1$ is isolated in $G_Y$. Then $Y\cap V(K_n^0) = \emptyset$ and $Y\cap V(^0K_n) = \emptyset$. Considering the vertices $(0,0)$ and $(0,j)$, where $j \in [n-1]$, we deduce that $|Y\cap V(K_n^j)| \ge 2$. As this holds for each such $j$, and since we have assumed that $n\ge 6$, we get $|Y| \ge 2(n-1) > \left\lceil \frac{4n}{3}\right\rceil$, a contradiction.

We next claim that no component of $G_Y$ is isomorphic to $K_2$ as soon as $n\ge 6$. Suppose on the contrary that this is the case and let, without loss of generality, the edge $00'$ induces a component of $G_Y$ isomorphic to $K_2$. This means that $(0,0)\in Y$ and that $Y\cap V(K_n^0) = \{(0,0)\}$ and $Y\cap V(^0K_n) = \{(0,0)\}$. Consider now an arbitrary edge of $G_Y$ different from $00'$. By the symmetry of $G$ we may without loss of generality assume that this additional edge is $11'$ (so that $(1,1)\in Y$). Consider now the vertices $(1,0)$ and $(0,1)$. Then $|Y\cap (V(K_n^{0}) \cup V(K_n^{1}) \cup V(^{0}K_n) \cup V(^{1}K_n))|\ge 4$. Because $Y\cap V(K_n^0) = \{(0,0)\}$ and $Y\cap V(^0K_n) = \{(0,0)\}$ it follows that the layers $K_n^{1}$ and $^{1}K_n$ together contain at least two vertices from $Y$ different from $(1,1)$. This in turn implies that $\deg_{G_Y}(1) + \deg_{G_Y}(1') \ge 4$. Therefore, the component of $G_Y$ containing the edge $11'$ has cardinality at least $4$. Let $q$ be the number of components of $G_Y$ different from the unique $K_2$ component, and let $n_i$, $i\in q$, be their cardinalities. As we have just proved, $n_i\ge 4$ holds for $i\in [q]$. Consequently, $q\le \left\lfloor \frac{2n-2}{4} \right\rfloor = \left\lfloor \frac{n-1}{2} \right\rfloor$. Since $2n = 2 + \sum_{i=1}^q n_i$, we can now estimate as follows:
\begin{align*}
|E(G_Y)| & \ge 1 + \sum_{i=1}^q(n_i-1) = 1 + (2n-2) - q\\
& \ge 2n - 1 - \left\lfloor \frac{n-1}{2} \right\rfloor\,.
\end{align*}
Since for $n\ge 6$ we have
$$|E(G_Y)| \ge 2n - 1 - \left\lfloor \frac{n-1}{2} \right\rfloor >
\left\lceil \frac{4n}{3} \right\rceil\,,$$
we can conclude that in this case $Y$ is not a weak $2$-metric basis.

Let now $n\ge 6$. Then by the above, each component of $G_Y$ is of cardinality at least $3$. Denoting the number of components  of $G_Y$ by $q$, and their respective cardinalities by $n_i$, $i\in [q]$, we have
$2n = \sum_{i=1}^q n_i$ and $q\le \left\lfloor \frac{2n}{3} \right\rfloor$. From these two estimates we can deduce that
\begin{align*}
|E(G_Y)| & \ge \sum_{i=1}^q(n_i-1) = 2n- q\\
& \ge 2n - \left\lfloor \frac{2n}{3} \right\rfloor \\
& = \left\lceil \frac{4n}{3} \right\rceil\,.
\end{align*}
Since $|E(G_Y)| = |Y|$, we can conclude that $\wdim_2(G) \ge \left\lceil \frac{4n}{3}\right\rceil$ which completes the formula for $k=2$.

\medskip
Once we have dealt with all the possible cases for $k$, the proof of Theorem \ref{thm:wdim-Kn-Kn-all} is completed.

%%%%%%%%%%%%%%%%%%%%%%%%%%%%%
%\section{ILP formulation and its simplification}\label{sec:ILP}
\section{ILP formulations for determining  $\wdim_{k}(G)$}\label{sec:ILP}
%%%%%%%%%%%%%%%%%%%%%%%%%%%%%

The problem of finding a weak $k$-metric basis in a graph $G$ can be stated as an integer linear programming problem with binary variables. The formulation of \cite{Peterin} associates every vertex set $S$ with a set of binary decision variables defined as follows:
$$s_u=\begin{cases}1; & u\in S\\ 0;& \text{otherwise}.\end{cases}$$
The formulation of~\cite{Peterin} is:
\begin{subequations}
\begin{align}
F_s\qquad \min\quad & \sum_{u\in V}s_u& &&\label{eq:ei:fo} \\
s.t. & \sum_{w\in V(G)}\left|d_G\left(u, w\right)-d_G\left(v, w\right)\right|s_w\geq k \qquad && u, v\in V(G), u<v \label{const}\\
s_u & \in\{0, 1\}, \forall u\in V(G).\label{domain}
\end{align}
\end{subequations}
Here we have assumed that the vertices of $G$ are linearly ordered by relation $<$. The number of constraints of $F_S$ is $\mathcal O(m^2\times n^2)$, which can be quite high even for instances with moderate values of $n$ and $m$.\\
For $G=K_{n}\cp K_{m}$, the coefficients $\left|d_G\left(u, w\right)-d_G\left(v, w\right)\right|$ can be precomputed for every triplet $u, v, w\in V$, with $u<v$. Let us use the notation $a_{uvw}=\left|d_G\left(u, w\right)-d_G\left(v, w\right)\right|$, and say that two vertices are \emph{aligned} if they agree in exactly one component, that is, if they belong to the same layer. Then,
$$a_{uvw}=\begin{cases}2; & \text{if $u$ and $v$ are not aligned, and either $w=u$ or $w=v$,}\\
%1; & \text{if  $u$ and $v$ are not aligned and exactly one of the vertices $u$ or $v$ is aligned with $w$},\\
1; & \text{if $u$ and $v$ are aligned and either $w=u$ or $w=v$},\\
1; & \text{if  $w\notin\{u,v\}$, and exactly one of the vertices $u$ or $v$ is aligned with $w$},\\ 0; & \text{otherwise.}
\end{cases}$$

Note that the only coefficients with value 2 appear in the Constraints \eqref{const} associated with pairs of vertices $u, v$ that are not aligned, for the indices $w\in\{u, v\}$, when $a_{uvu}=a_{uvv}=2$. The other non-zero coefficients appear when $w$ is aligned with exactly one of the vertices $u$ or $v$, and it is possible (but not necessary) that $u$ and $v$ are aligned. For ease of presentation, for a given vertex pair $u, v$, $v>u$, we will use the notation $I_{uv}$ to denote the index set of the vertices $w$ such that $a_{uvw}=1$, that is,
$$I_{uv}=\{w\in V:\ w \text{ is aligned with exactly one of the vertices}\ u\ \text{or}\ v\}\,.$$
Then, the set of constraints \eqref{const} can be rewritten as:

\begin{subequations}
\begin{align}
      & s_u + s_v + \sum_{w\in I_{uv}}s_w\geq k \qquad && u, v\in V(G),\, u<v \text{ aligned,} \label{const:align}\\
     & 2 s_u + 2 s_v + \sum_{w\in I_{uv}}s_w\geq k \qquad && u, v\in V(G), \, u<v \text{ not aligned.} \label{const:noalign}
\end{align}
\end{subequations}

Observe that the set of constraints \eqref{const:align} is $\mathcal O(m\times n)$ whereas the set of constraints \eqref{const:noalign} remains $\mathcal O(m^2\times n^2)$.\\

Next we see that for values of $k\geq 4$ the set of Constraints \eqref{const:noalign} associated with non-aligned pairs of vertices are not needed, as they are implied by those associated with aligned vertices. This will allow us to ignore this set of constraints and work with a formulation having only $\mathcal O(m\times n)$ constraints. The formulation in which the set of constraints \eqref{const} is substituted by \eqref{const:align} will be referred to as $F_S^-$.

We first observe that  for any pair of non-aligned vertices $u, v\in V(G)$ there exist exactly two vertices $\tilde u, \tilde v\in V$ that are aligned with both $u$ and $v$.

\begin{proposition}
If $k\geq 4$, then the set of Constraints \eqref{const:noalign} associated with non-aligned pairs are implied by the set of Constraints \eqref{const:align} associated with pairs of aligned vertices.
\end{proposition}

\begin{proof}
Let us assume that $k\geq 4$ and let $\widehat s$ be a binary solution satisfying Constraints \eqref{const:align}. Consider a given pair of non-aligned vertices $u, v\in V(G)$, $u<v$, and let us see that the associated Constraint \eqref{const:noalign} is also satisfied by $\widehat s$.\\

Let $\tilde u, \tilde v\in V(G)$ be the two vertices  that are aligned with both $u$ and $v$.
Indeed, $$I_{uv}=\left(I_{u\tilde v}\setminus \{\tilde u, v\}\right)\cup\left(I_{u\tilde u}\setminus\{\tilde v, v\}\right)\,,$$ and the subsets $\left(I_{u\tilde v}\setminus \{\tilde u, v\}\right)$ and $\left(I_{u\tilde u}\setminus\{\tilde v, v\}\right)$ are disjoint (see Fig.~\ref{fig:coefs}).

\begin{figure}[ht]
    \centering
      %\begin{subfigure}[b]{0.3\textwidth}
   \begin{subfigure}%{0.3\textwidth}
         \centering
         \includegraphics[width=0.32\textwidth]{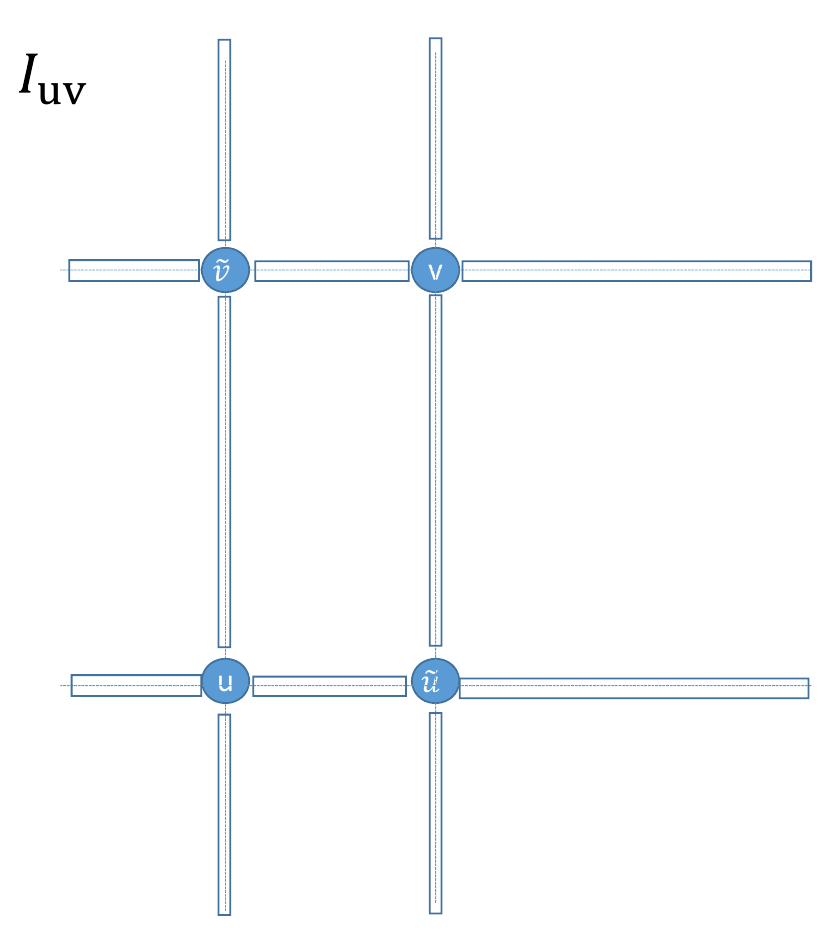}
         %\caption{Transformed graph $G_T$}
         \label{fig:RPPtransformedgraph}
    \end{subfigure}
    \begin{subfigure}%{0.3\textwidth}
         \centering
         \includegraphics[width=0.32\textwidth]{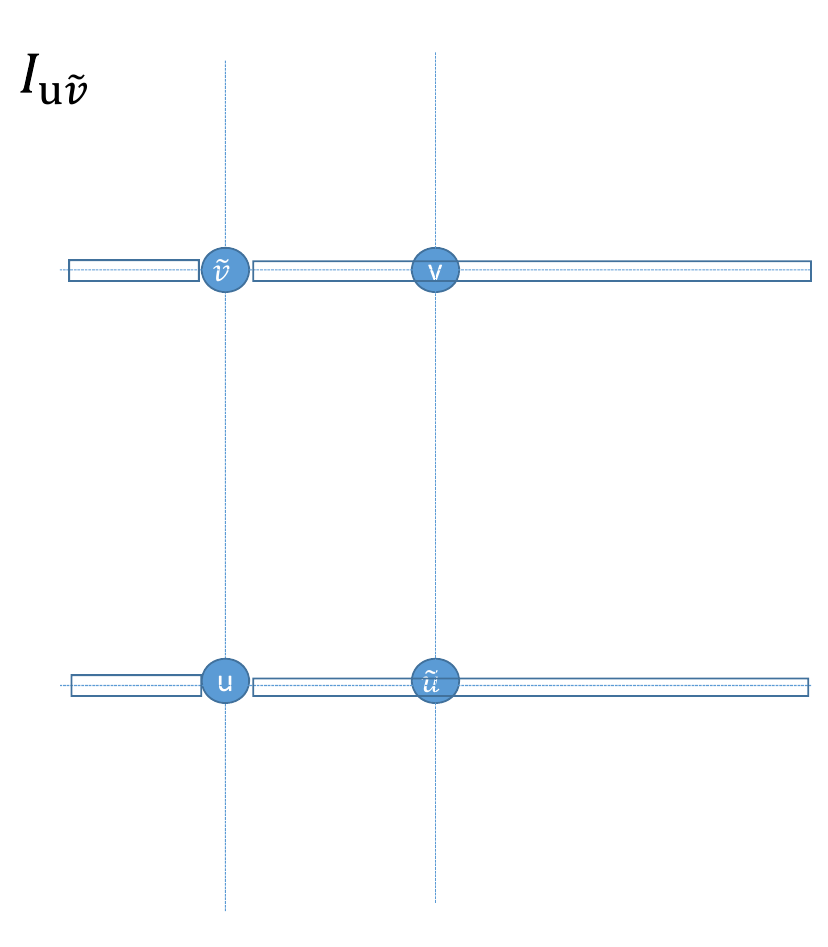}
         %\caption{An infeasible solution}
         \label{fig:RPPinfeasiblesolution}
    \end{subfigure}
    \begin{subfigure}%{0.3\textwidth}
         \centering
         \includegraphics[width=0.32\textwidth]{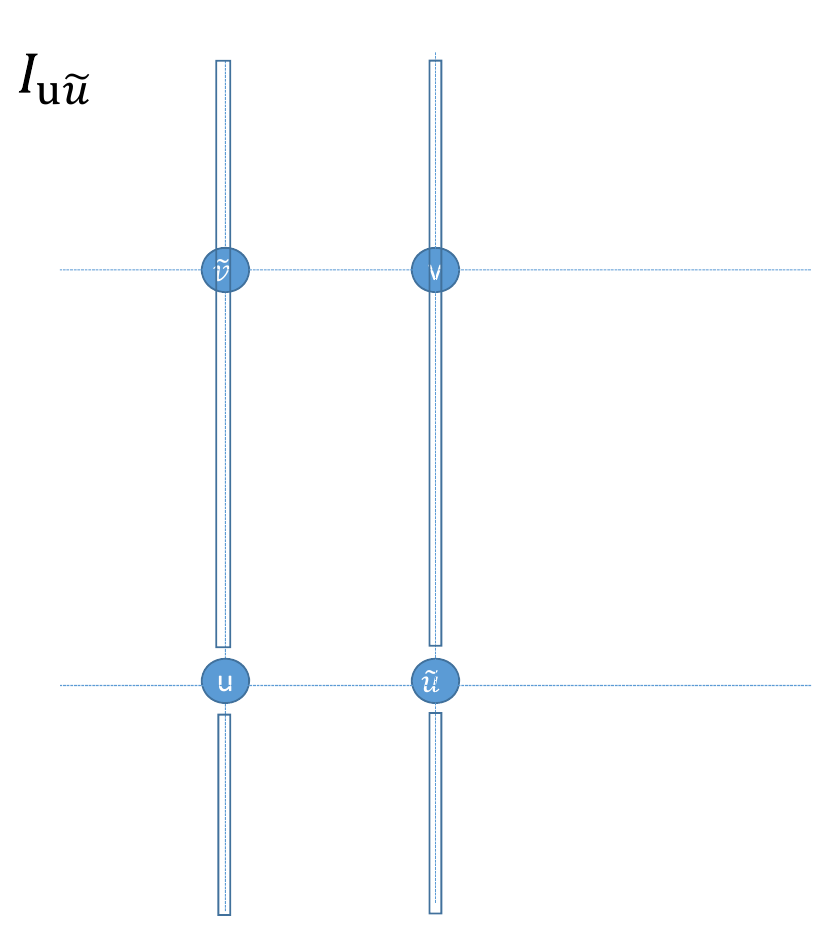}
         %\caption{Multigraph $(V, \overline E)$}
         \label{fig:RPPmultigraph}
    \end{subfigure}
        \caption{Definition of sets $I_{uv}$, $I_{u\tilde v}$, and $I_{u\tilde u}$}
        \label{fig:coefs}
\end{figure}

Hence, $$\sum_{w\in I_{uv}}\widehat s_w= \sum_{w\in I_{u\tilde v}}\widehat s_w+\sum_{w\in I_{u\tilde u}}\widehat s_w- 2 \widehat s_v-\widehat s_{\tilde u}-\widehat s_{\tilde v}.$$

Therefore
\begin{align*}
2 \widehat s_u + 2 \widehat s_{v} + \sum_{w\in I_{u v}}\widehat s_w &  = 2 \widehat s_u + 2 \widehat s_{v} +\left(\sum_{w\in I_{u\tilde v}}\widehat s_w-\widehat s_{\tilde u}-\widehat s_{v}\right)+\left(\sum_{w\in I_{u\tilde u}}\widehat s_w-\widehat s_v-\widehat s_{\tilde v}\right)\\
& = 2 \widehat s_u-  \widehat s_{\tilde u}- \widehat s_{\tilde v}+\sum_{w\in I_{u\tilde v}}\widehat s_w+\sum_{w\in I_{u\tilde u}}\widehat s_w\,.
 %\sum_{w\in I_{u\tilde u}}\widehat s_w  \geq 2 \widehat s_u + \widehat s_{\tilde u} + \widehat s_{\tilde v}+ \sum_{w\in I_{u\tilde v}}\widehat s_w+ \sum_{w\in I_{u\tilde u}}\widehat s_w\geq 2 k \qquad $$
\end{align*}

Since the constraints \eqref{const:align} associated with the pair $u$, $\tilde v$, and with the pair $u$, $\tilde u$ are both satisfied by $\widehat s$, we have:

\begin{align}
      & \widehat s_{u} + \widehat s_{\tilde v} + \sum_{w\in I_{u\tilde v}}\widehat s_w\geq k \quad && \Rightarrow \quad \widehat s_{u} +  \sum_{w\in I_{u\tilde v}}\widehat s_w\geq k-1 \nonumber \\
      & \widehat s_u + \widehat s_{\tilde u} + \sum_{w\in I_{u \tilde u}}\widehat s_w\geq k \quad && \Rightarrow \quad \widehat s_{u} +  \sum_{w\in I_{u \tilde u}}\widehat s_w\geq k-1. \nonumber
\end{align}

Therefore
$$ 2 \widehat s_u + 2 \widehat s_{v} + \sum_{w\in I_{u v}}\widehat s_w= 2 \widehat s_u +  \sum_{w\in I_{u\tilde v}}\widehat s_w+ \sum_{w\in I_{u\tilde u}}\widehat s_w  -  \widehat s_{\tilde u}- \widehat s_{\tilde v}  \geq 2 \left(k-1\right)-2\geq k \Leftrightarrow k\geq 4.$$
\end{proof}

Preliminary computational testing showed that formulation  $F_s$ can be quite time consuming. For fixed values of $n$ and $m$, it was especially  time consuming for small values of $k\in\{2, 3\}$.
As could be expected, for instances with the same values of $n$, $m$, and $k\geq 4$, formulation $F_s^-$ outperformed $F_s$.
Still, despite the reduction in the number of constraints, formulation $F_S^-$ for $k\geq 4$ can also be quite time consuming.
Both formulations usually produce optimal or near-optimal solutions in very small computing times, although proving the optimality of such solutions may take quite high computing times.
This is particularly true for instances with odd values of the parameter $k$. For example, for $K_5\cp K_m$ with $2\le k\le 10$, all instances with $k$ even can be solved to proven optimality within a time limit of 600 seconds, whereas with $k$ odd no instance can be solved to proven optimality, even if the time limit is increased to 7,200 seconds.
For this reason we next develop an alternative formulation, which computationally performs notably better, even if it requires more decision  variables than formulations $F_S$ and $F_S^-$.\\ %\eqref{eq:ei:fo} -\eqref{domain}.\\

In addition to the original decision variables $s_u$, $u\in V(G)$ we use additional decision variables to denote the number of elements of $S$ in each horizontal and vertical layer, namely:

\begin{itemize}
\item $h_j=\text{Number of elements of $S$ in the horizontal layer }K_n^j,\, j\in V(K_m)$.
\item $g_i=\text{Number of elements of $S$ in the vertical layer } ^iK_m,\ i\in V(K_n)$.
\end{itemize}
%In order to relate variables $\mathbf{g}$ and $\mathbf{h}$ with the original variables $\mathbf{s}$, we use the following set of auxiliary binary variables, which are defined forall $1\i\leq n$, $1\leq j\leq m$:

%$$z_{ij}=\begin{cases}1; & \text{the vertex of coordinates } (i, j)\in S\\ 0;& \text{otherwise}.\end{cases}$$

Indeed, we have
\begin{subequations}
\begin{align}
& h_j = \sum_{u\in V(K_n^j)}s_u \qquad\qquad && j\in V(K_m)\nonumber\\
& g_i = \sum_{u\in V(^iK_m)}s_u \qquad\qquad && i\in V(K_n).\nonumber
\end{align}
\end{subequations}

In the following, for a given pair of vertices $u, v\in V(G)$, we make explicit their coordinates with the notation  $u=(i_u, j_u)$,  $v=(i_v, j_v$). Moreover, when $u$ and $v$ are not aligned we will also use the notation  $\tilde u=(i_u, j_v)$ and $\tilde v=(i_v, j_u)$.

\begin{proposition}
The set of constraints \eqref{const:align}-\eqref{const:noalign} can be expressed in terms of the variables $\mathbf{h}$ and $\mathbf{g}$ as:
\begin{subequations}
\begin{align}
      & g_{i_u}+g_{i_v}\geq k && u,\, v\in V(G),\, u<v \text{ aligned horizontally}\label{const-aligh-g}\\
      & h_{j_u}+h_{j_v}\geq k && u,\, v\in V(G),\, u<v  \text{ aligned vertically}\label{const-aligh-h}\\
      & h_{j_u}+h_{j_v}+g_{i_u}+g_{i_v}- 2s_{\tilde u}-2s_{\tilde v}\geq k \quad && u,\, v\in V(G),\, u<v  \text{ not aligned.}\label{const-notaligh2}
\end{align}
\end{subequations}
\end{proposition}
\begin{proof}
Let $u, v\in V(G)$, $u<v$, with $u=(i_u, j_u)$, $v=(i_v, j_v)$, and consider the following cases:

\medskip\noindent
\textbf{Case 1:} $u,\ v$ are aligned.\\
%for every pair of aligned vertices $u, v\in V(G)$, $u<v$,
For constraints \eqref{const:align} we distinguish the following subcases:
\begin{itemize}
\item $u, v\in {V(^iK_m)}$ for some $i\in V(K_n)$. %$1\leq i\leq n$.
Then, $a_{uvw}=0$  for all $w\in {V(^iK_m)}\setminus \{u, v\}$, i.e.,
$V(^iK_m)\cap  I_{uv}=\emptyset$,  so $V(K_n^{j_u})\cup  V(K_n^{j_v})=\{u, v\}\cup  I_{uv}$.\\ %, where $j_u$, $j_v$ are the indices of the vertical layers of $u$ and $v$, respectively.\\
Since $V(K_n^{j_u})\cap  V(K_n^{j_v})=\emptyset$, we have that $s_u+s_v+\sum_{w\in I_{uv}}s_w=h_{j_u}+h_{j_v}$ so  the constraint \eqref{const:align} associated with the pair $u,\,v$ can be rewritten as:
$$h_{j_u}+h_{j_v}\geq k.$$
\item $u, v\in V(K_n^j)$ for some $j\in V(K_m)$. %$1\leq j\leq m$.
Now,  $a_{uvw}=0$  for all $w\in V(K_n^j)\setminus \{u, v\}$, i.e.,
$V(K_n^j)\cap  I_{uv}=\emptyset$,  so $V({^{i_u}K_m}) \cup V({^{i_v}K_m}) =\{u, v\}\cup  I_{uv}$.\\ %, where $i_u$, $i_v$ are the indices of the horizontal layers of $u$ and $v$, respectively.\\
Again, $V({^{i_u}K_m}) \cap V({^{i_v}K_m})=\emptyset$ and we have that $s_u+s_v+\sum_{w\in I_{uv}}s_w= g_{i_u}+g_{i_v}$ so  the constraint \eqref{const:align} associated with the pair $u,\,v$ can be rewritten as:
$$g_{i_u}+g_{i_v}\geq k.$$
\end{itemize}
\noindent
\textbf{Case 2:} $u,\ v$ are not aligned.\\
In this case, a similar analysis can be applied to the constraints \eqref{const:noalign}, which are needed for $k<4$. Now, with two exceptions, the vertices of $I_{uv}$ are those of the vertical layers ${^{i_u}}K_m$ and ${^{i_v}}K_m$
 and the horizontal layers $K_n^{j_u}$ and $K_n^{j_v}$. The exceptions are the vertices $\tilde u=(i_u, j_v)$ and $\tilde v=(i_v, j_u)$, which appear in $V({K_n^{j_v}})\cap V({^{i_u}}K_m)$ and $V({K_n^{j_u}}) \cap V({^{i_v}}K_m)$, respectively. These two vertices are at distance $1$ from both $u$ and $v$ so they should not appear in the constraint.
Moreover, the two vertices that appear with coefficient 2 in the constraint \eqref{const:noalign} are precisely  vertex $u=(i_u, j_u)\in V({K_n^{j_u}})\cap V({^{i_u}K_m})$ and the vertex $v=(i_v, j_v)\in V({K_n^{j_v}}) \cap V({^{i_v}K_m})$.
Therefore, $h_{j_u}+h_{j_v}+g_{i_u}+g_{i_v}- 2 s_{\tilde u}-2 s_{\tilde v}= 2 s_u+2s_v+\sum_{w\in I_{uv}} s_w$, so the constraint \eqref{const:noalign} associated with the pair $u,\,v$ can be rewritten as:
    $$h_{j_u}+h_{j_v}+g_{i_u}+g_{i_v}- 2 s_{\tilde u}-2 s_{\tilde v}\geq k,$$
which completes our proof.
\end{proof}

Furthermore, we observe that the sets of constraints \eqref{const-aligh-g} and \eqref{const-aligh-h} can be reduced notably. In particular, let $u, v\in V(G)$ and $u', v'\in V(G)$ be two pairs of horizontally aligned vertices in the same two vertical layers, i.e.
$(i)$\, $j_u=j_v$ and $j_{u'}=j_{v'}$;  $(ii)$\, $u,u'\in V(^{i}K_m)$ for some $i\in V(K_n)$ and $v,v'\in V(^{i'}K_m)$ for some $i'\in V(K_n)$ with $i\ne i'$.
Then, the constraints \eqref{const-aligh-g} associated with the pair $u,\, v$ and with the pair $u',\, v'$ are exactly the same. This means that the set of constraints \eqref{const-aligh-g} reduces to only single constraint for every pair of vertical layers $i, i'\in  V(K_n)$, $i\ne i'$. \\
%Furthermore, we observe that the sets of constraints \eqref{const-aligh-g} and \eqref{const-aligh-h} can be reduced notably. In particular, let $u, v\in V(G)$ and $u', v'\in V(G)$ two pairs of horizontally aligned vertices in the same two vertical layers, i.e. {\color{blue}$i_{u}=i_{v}$, $i_{u'}=i_{v'}$ with $j_u\ne j_{u'}$ and $i_{u'}=i_{u}$, $i_{v'}=i_{v}$}. Then, the constraints \eqref{const-aligh-g} associated with the pair $u,\, v$ and with the pair $u',\, v'$ are exactly the same. This means that the set of constraints \eqref{const-aligh-g} reduces to only single constraint for every pair of vertical layers $j, j'$, $j\ne j'$. \\

Similarly, the set of constraints \eqref{const-aligh-h} reduces to only single constraint for every pair of horizontal layers $j, j'\in  V(K_m)$, $j\ne j'$. \\

This observation can be summarized in the result below:

\begin{corollary}
\label{coro:last-ILP}
The following integer linear programming formulation produces a weak $k$-metric basis in a graph $G$:

\begin{subequations}
\begin{align}
F_{gh}\qquad \min\quad & \sum_{u\in V(G)}s_u &&\label{eq:ei:fo2} \\
s.t.\quad  & h_j = \sum_{u\in V(K_n^j)}s_u \qquad\qquad && j\in V(K_m)\label{const-h}\\
      & g_i = \sum_{u\in V(^iK_m)}s_u \qquad\qquad && i\in V(K_n)\label{const-g}\\
      & h_{j}+h_{j'}\geq k && j, j'\in V(K_m),\, j < j'\label{const-aligh-hg}\\
      & g_{i}+g_{i'}\geq k && i, i'\in V(K_n),\, i < i'\label{const-aligh-hh}\\      & h_{j_u}+h_{j_v}+g_{i_u}+g_{i_v}- 2s_{\tilde u}-2s_{\tilde v}\geq k \quad && u,\, v\in V(G),\, u<v  \text{ not aligned}\label{const-notaligh22}\\
s_u & \in\{0, 1\}, \forall u\in V(G)\label{domain2}\\
h_j,\, & g_i \text{ integer, } \forall j\in V(K_m), i\in V(K_n).\label{domain22}
\end{align}
\end{subequations}

Formulation $F_{gh}$ has $m$ constraints \eqref{const-h} and $n$ constraints \eqref{const-g}. The number of constraints of each of the sets \eqref{const-aligh-hg} and \eqref{const-aligh-hh} is $\mathcal O(m^2)$ and $\mathcal O(n^2)$, respectively. Finally, it has $\mathcal O(n^2m^2)$ constraints \eqref{const-notaligh22}, which, as in formulation $F_s$, are only needed when $k\leq 3$.\\
%As in the previous formulation, the constraints \eqref{const-notaligh22} are only needed when $k\leq 3$.\\

Note that for $k\geq 4$ $F_{gh}$ admits the following (simple) interpretation: every pair of vertical layers must contain at least $k$ elements of $S$ and every pair of horizontal layers must contain at least $k$ elements of $S$.
\end{corollary}

\section{Computational Experiments}

In order to analyze the empirical performance of formulation $F_{gh}$, %\eqref{eq:ei:fo2}- \eqref{domain22},
we have carried out a series of computational experiments. The objective of these experiments is essentially to analyze the structure of the solutions it produces, so as to serve as an empirical support for the Hamming graphs for which theoretical results are not yet known. We also analyze the effectiveness and scalability of the formulation.\\

All the computational tests have been carried out in an AMD Ryzen 7 PRO 2700U 2.20 GHz with 8 GB RAM, under Windows 10 Pro as operating system. Formulation $F_{gh}$ has been coded in Mosel 5.4.1 %5.6.0
using as solver Xpress Optimizer \cite{Xpress}.
%38.01.01  \cite{Xpress}.
For the experiments we have considered the following sets of benchmark instances for the two-dimensional Hamming graphs $K_n\cp K_m$:

\medskip\noindent
\textbf{CE$_5$}: $K_5\cp K_m$ for $2\le k\le 11$ and $5\le m\le 20$.

\smallskip\noindent
\textbf{CE$_6$:}  $K_6\cp K_m$ for $2\le k\le 11$ and $5\le m\le 20$.

\smallskip\noindent
\textbf{CE$_7$:}  $K_7\cp K_m$ for $2\le k\le 11$ and $5\le m\le 20$.

\smallskip\noindent
\textbf{CE$_8$:} $K_8\cp K_m$ for $2\le k\le 16$ and $8\le m\le 20$.

%These instances have $5m$ binary $\matbf s$ variables and $m+n$
\medskip
A computing time limit of 600 seconds has been set for each solved instance. The results of formulation $F_{gh}$ for the different groups of instances are reported in Tables~\ref{tab1}-\ref{tab2} for the sets \textbf{CE$_5$}-\textbf{CE$_8$}, respectively. In all tables, the first row indicates the number of horizontal layers (value of the parameter $m$), and the first column shows the value of the parameter $k$. All other entries indicate the value of $\wdim_k(K_n\cp K_m)$ for the instance with the corresponding parameters.

The tables do not include the computing times as all instances could be optimally solved in very small computing times.
In particular, for values of $k\geq 4$ all instances are solved within less than five seconds. The most time-consuming ones are those with $k\in\{2, 3\}$, where the size of the formulation increases due to the constraints \eqref{const-notaligh22} for non-aligned vertices. Still, all tested instances could be solved in less than 60 seconds.\\
%The most time consuming instance (after $K_8xK_8$ with $k=3$) is $K_8xK_25$ with $k=3$, which takes 72 seconds.\\

As can be seen, with the exception of the instances with $k\in\{2, 3\}$, the optimal values of the tested instances follow a very specific pattern, which depends on the parameter values. In all the cases, the obtained results support the validity of Conjecture \ref{conject}.

\renewcommand{\arraystretch}{1.5} % Default value: 1
\begin{table}[ht!]
\centering
\begin{tabular}{|c||c|c|c|c|c|c|c|c|c|c|c|c|c|c|c|c|}
\hline
$k\backslash m$ & $5$ & $6$ & $7$ & $8$ & $9$ & $10$ & $11$ & $12$ & $13$ & $14$ & $15$ & $16$ & $17$ & $18$ & $19$ & $20$ \\
\hline\hline
 2 &  7 &  8 &  8 &  9 & 10 & 10 & 11 & 12 & 13 & 14 & 15 & 16 & 17 & 18 & 19 & 20 \\
\hline
 3 & 10 & 11 & 13 & 15 & 17 & 19 & 21 & 23 & 25 & 27 & 29 & 31 & 33 & 35 & 37 & 39 \\
\hline
 4 & 10 & 12 & 14 & 16 & 18 & 20 & 22 & 24 & 26 & 28 & 30 & 32 & 34 & 36 & 38 &  40\\
\hline
 5 & 14 & 17 & 20 & 23 & 26 & 29 & 32 & 35 & 38 & 41 & 44 & 47 & 50 & 53 & 56 & 59 \\
\hline
 6 & 15 & 18 & 21 & 24 & 27 & 30 & 33 & 36 & 39 & 42 & 45 & 48 & 51 & 54 & 57 & 60\\
\hline
 7 & 19 & 23 & 27 & 31 & 35 & 39 & 43 & 47 & 51 & 55 & 59   & 63 & 67 & 71 & 75 & 79\\
\hline
 8 & 20 & 24 & 28 & 32 & 36 & 40 & 44 & 48 & 52   & 56 & 60 & 64 & 68 & 72 & 76 & 80\\
\hline
 9 & 24 & 29 & 34 & 39 & 44 & 49 & 54 & 59 & 64 & 69 & 74 & 79 & 84 & 89 & 94 & 99\\
\hline
10 & 25 & 30 & 35 & 40 & 45 & 50 & 55 & 60 & 65 & 70 & 75 & 80 & 85 & 90 & 95 & 100\\
\hline
%11 &  NF  & NF   & NF   & NF   &  NF  &  NF  & NF & NF & NF & NF & NF & NF & NF & NF  & NF & NF \\
%\hline
\end{tabular}
\caption{Values of $\wdim_k(K_5\cp K_m)$ for $2\le k\le 11$ and $5\le m\le 20$ as computed by $F_{gh}$.}
\label{tab1}
\end{table}

\renewcommand{\arraystretch}{1.5} % Default value: 1
\begin{table}[ht!]
\centering
\begin{tabular}{|c||c|c|c|c|c|c|c|c|c|c|c|c|c|c|c|}
\hline
$k\backslash m$ & $6$ & $7$ & $8$ & $9$ & $10$ & $11$ & $12$ & $13$ & $14$ & $15$ & $16$ & $17$ & $18$ & $19$ & $20$ \\
\hline\hline
 2 & 8 &  9 & 10 & 10 & 10 & 11 & 12 & 13 & 14 & 15 & 16 &  17 &  18 &  19 &  20  \\
 \hline
 3 & 12 & 13 & 15 & 17 & 19 & 21 & 23 & 25 & 27 & 29 & 31 &  33 &  35 &  37 &  39  \\
 \hline
 4 & 12 & 14 & 16 & 18 & 20 & 22 & 24 & 26 & 28 & 30 & 32 &  34 &  36 &  38 &  40 \\
 \hline
 5 & 17 & 20 & 23 & 26 & 29 & 32 & 35 & 38 & 41 & 44 & 47 &  50 &  53 &  56 &  59  \\
 \hline
 6 & 18 & 21 & 24 & 27 & 30 & 33 & 36 & 39 & 42 & 45 & 48 &  51 &  54 &  57 &  60 \\
 \hline
 7 & 23 & 27 & 31 & 35 & 39 & 43 & 47 & 51 & 55 & 59 & 63 &  67 &  71 &  75 &  79 \\
 \hline
 8 & 24 & 28 & 32 & 36 & 40 & 44 & 48 & 52 & 56 & 60 & 64 &  68 &  72 &  76 &  80 \\
 \hline
 9 & 29 & 34 & 39 & 44 & 49 & 54 & 59 & 64 & 69 & 74 & 79 &  84 &  89 &  94 &  99 \\
\hline
10 & 30 & 35 & 40 & 45 & 50 & 55 & 60 & 65 & 70 & 75 & 80 &  85 &  90 &  95 & 100 \\
\hline
11 & 35 & 41 & 47 & 53 & 59 & 65 & 71 & 77 & 83 & 89 & 95 & 101 & 107 & 113 & 119 \\
\hline
12 & 36 & 42 & 48 & 54 & 60 & 66 & 72 & 78 & 84 & 90 & 96 & 102 & 108 & 114 & 120 \\
\hline
\end{tabular}
\caption{Values of $\wdim_k(K_6\cp K_m)$ for $2\le k\le 12$ and $6\le m\le 20$ as computed by $F_{gh}$.}
\label{tab-n=6}
\end{table}

\renewcommand{\arraystretch}{1.5} % Default value: 1
\begin{table}[ht!]
\centering
\begin{tabular}{|c||c|c|c|c|c|c|c|c|c|c|c|c|c|c|}
\hline
$k\backslash m$ & $7$ & $8$ & $9$ & $10$ & $11$ & $12$ & $13$ & $14$ & $15$ & $16$ & $17$ & $18$ & $19$ & $20$ \\
\hline\hline
 2 & 10 & 10 & 11 & 12 & 12 & 12 & 13 & 14 &  15 &  16 &  17 &  18 &  19 &  20   \\
\hline
 3 & 14 & 15 & 17 & 19 & 21 & 23 & 25 & 27 &  29 &  31 &  33 &  35 &  37 &  39   \\
\hline
 4 & 14 & 16 & 18 & 20 & 22 & 24 & 26 & 28 &  30 &  32 &  34 &  36 &  38 &  40  \\
\hline
 5 & 20 & 23 & 26 & 29 & 32 & 35 & 38 & 41 &  44 &  47 &  50 &  53 &  56 &  59   \\
\hline
 6 & 21 & 24 & 27 & 30 & 33 & 36 & 39 & 42 &  45 &  48 &  51 &  54 &  57 &  60  \\
\hline
 7 & 27 & 31 & 35 & 39 & 43 & 47 & 51 & 55 &  59 &  63 &  67 &  71 & 75 &  79  \\
\hline
 8 & 28 & 32 & 36 & 40 & 44 & 48 & 52 & 56 &  60 &  64 &  68 &  72 &  76 &  80  \\
\hline
 9 & 34 & 39 & 44 & 49 & 54 & 59 & 64 & 69 &  74 &  79 &  84 &  89 &  94 &  99  \\
\hline
10 & 35 & 40 & 45 & 50 & 55 & 60 & 65 & 70 &  75 &  80 &  85 &  90 &  95 & 100  \\
\hline
11 & 41 & 47 & 53 & 59 & 65 & 71 & 77 & 83 &  89 &  95 & 101 & 107 & 113 & 119  \\
\hline
12 & 42 & 48 & 54 & 60 & 66 & 72 & 78 & 84 &  90 &  96 & 102 & 108 & 114 & 120  \\
\hline
13 & 48 & 55 & 62 & 69 & 76 & 83 & 90 & 97 & 104 & 111 & 118 & 125 & 132 & 139  \\
\hline
14 & 49 & 56   & 63   & 70   &  77  & 84   &  91  & 98   &   105  &  112   &  119   &  126   &   133  &    140  \\
\hline
\end{tabular}
\caption{Values of $\wdim_k(K_7\cp K_m)$ for $2\le k\le 14$ and $7\le m\le 20$ as computed by $F_{gh}$.}
\label{tab-n=7}
\end{table}

\renewcommand{\arraystretch}{1.5} % Default value: 1
\begin{table}[ht!]
\centering
%\scriptsize
\begin{tabular}{|c||c|c|c|c|c|c|c|c|c|c|c|c|c|c|c|c|c|c|c|}
\hline
$k\backslash m$ &	8 &	9 &	10 &	11 &	12 &	13 &	14 &	15 &	16 &	17 &	18 &	19 &	20\\             \hline\hline
 2	&	11	&	12	&	12	&	13	&	14	&	14 	&	 14	&	 16	&	 16	&	 17	&	 18	&	 19	&	 20	\\
\hline
3	&	16	&	17	&	19	&	21	&	23	&	25 	&	 27	&	 29	&	 31	&	 33	&	 35	&	 37	&	 39	 \\
\hline
4	&	16	&	18	&	20	&	22	&	24	&	26 	&	 28	&	 30	&	 32	&	 34	&	 36	&	 38	&	 40	\\
\hline
5	&	23	&	26	&	29	&	32	&	35	&	38 	&	 41	&	 44	&	 47	&	 50	&	 53	&	 56	&	 59 \\
\hline
6	&	24	&	27	&	30	&	33	&	36	&	39 	&	 42	&	 45	&	 48	&	 51	&	 54	&	 57	&	 60\\
\hline
7	&	31	&	35	&	39	&	43	&	47	&	51 	&	 55	&	 59	&	 63	&	 67	&	 71	&	 75	&	 79	\\
\hline
8	&	32	&	36	&	40	&	44	&	48	&	52 	&	 56	&	 60	&	 64	&	 68	&	 72	&	 76	&	 80	\\
\hline
9	&	39	&	44	&	49	&	54	&	59	&	64 	&	 69	&	 74	&	 79	&	 84	&	 89	&	 94	&	 99\\
\hline
10	&	40	&	45	&	50	&	55	&	60	&	65 	&	 70	&	 75	&	 80	&	 85	&	 90	&	 95	&	100	\\
\hline
11	&	47	&	53	&	59	&	65	&	71	&	77 	&	 83	&	 89	&	 95	&	101	&	107	&	113	&	119\\
\hline
12	&	48	&	54	&	60	&	66	&	72	&	78 	&	 84	&	 90	&	 96	&	102	&	108	&	114	&	120	 \\
\hline
13	&	55	&	62	&	69	&	76	&	83	&	90 	&	 97	&	104	&	111	&	118	&	125	&	132	&	139	\\
\hline
14	&	56	&	63	&	70	&	77	&	84	&	91 	&	 98	&	105	&	112	&	119	&	126	&	133	&	140	\\
\hline
15	&	63	&	71	&	79	&	87	&	95	&	103	&	111	&	119	&	127	&	135	&	143	&	151	&	159	\\
\hline
16	&	64	&	72	&	80	&	88	&	96	&	104	&	112	&	120	&	128	&	136	&	144	&	152	&	160	\\
\hline
\end{tabular}
\caption{Values of $\wdim_k(K_8\cp K_m)$ for $2\le k\le 16$ and $8\le m\le 20$ as computed by $F_{gh}$.}
\label{tab2}
\end{table}

\section{Concluding remarks}

\begin{itemize}
    \item Based on the computational results from Section \ref{sec:ILP}, the conclusion of Corollary \ref{coro:last-ILP}, and the formulas from Theorem \ref{thm:wdim-Kn-Kn-all}, we pose the following conjecture, whose proof might be done by a tedious and lengthy considerations similar to those ones used in the proof of Theorem \ref{thm:wdim-Kn-Kn-all}.

    \begin{conjecture}\label{conject}
    If $n\ge 3$, $m\ge n+1$ and $3\le k\le 2n$, then
$$
\wdim_{k}(K_n\cp K_m) =
\begin{cases}
    m\left\lceil\frac{k}{2}\right\rceil; & \mbox{if $k$ is even}\,, \\[0.2cm]
    m\left\lceil\frac{k}{2}\right\rceil-1; & \mbox{if $k$ is odd}\,.
\end{cases}
$$
    \end{conjecture}
    Moreover, we strongly believe that the formula from the conjecture above is also valid when $k=2$ and $m\ge 2n$. Notice that when $n=5$, the formula does not hold for $m\in\{6,\dots,9\}$, and for $n=6$, the formula does not hold for $m\in\{7,\dots,11\}$.
    \item We have computed in this work the weak $k$-metric dimension of the two-dimensional Hamming graph $K_n\cp K_n$ for any $n\ge 3$ and $2\le k\le 2n$. A natural continuation of this work will be that of considering the weak $k$-metric dimension of the $d$-dimensional Hamming graph $K_n^d$, $n\ge 2$, for the suitable values of $k$ given in Proposition \ref{prop:kappa=min-Delta}. In particular, it would be desirable to compute the value of such a parameter for the hypercube graph $Q_d$ for any large enough integer $d$ and $1\le k\le 2^d$.
\end{itemize}

\section*{Acknowledgement}

E. Fern\'andez, D. Kuziak, M. Mu\~noz-M\'arquez and I.G.\ Yero have been partially supported by ``Ministerio de Ciencia, Innovaci\'on y Universidades'' through the grant PID2023-146643NB-I00, and Cadiz University Research Program.
Sandi Klav\v zar was supported by the Slovenian Research and Innovation Agency (ARIS) under the grants P1-0297, N1-0355, and N1-0285.
D.\ Kuziak also acknowledges the support from ``Ministerio de Educaci\'on, Cultura y Deporte'', Spain, under the ``Jos\'e Castillejo'' program for young researchers (reference number: CAS22/00081) to make a temporary visit to the University of Ljubljana, where this investigation has been partially developed.

\end{document}